\documentclass[hidelinks,onefignum,onetabnum%
]{siamart250211}     

\usepackage{color}
\usepackage{cite}
\usepackage{amsfonts}
\usepackage{amsmath}
\usepackage{amssymb}
\usepackage{array}
\usepackage{stmaryrd}
\usepackage{graphicx}
\usepackage{color}
\usepackage{enumitem}
\usepackage{epstopdf}
\usepackage[caption=false]{subfig}
\usepackage{float}

\newsiamremark{remark}{Remark} 

\numberwithin{equation}{section}

\title{Pointwise-in-time error bounds for semilinear and quasilinear fractional subdiffusion equations \\on graded meshes\thanks{%
This research was conducted with the financial support of Science Foundation Ireland under Grant number 18/CRT/6049.}}

\author{Se{\'a}n Kelly\thanks{Department of Mathematics and Statistics, University of Limerick, Ireland.
(\email{sean.f.kelly@ul.ie}).}
\and Natalia Kopteva\thanks{Department of Mathematics and Statistics, University of Limerick, Ireland.
(\email{natalia.kopteva@ul.ie}).}}

\definecolor{blue}{rgb}{0,0,0.6}
\definecolor{ForestGreen}{rgb}{0,0.6,0}
\newcommand{\pt}{\partial}
\newcommand{\R}{\mathbb{R}}

\renewcommand{\sim}{\simeq}
\newcommand{\LL}{{\mathcal L}}
\newcommand{\QQ}{{\mathcal Q}}
\newcommand{\B}{{\mathcal B}}

\newcommand {\U} {{\mathcal U}}
\newcommand {\RR} {{\mathcal R}}

\newcommand {\beq} {\begin{equation}}
\newcommand {\eeq} {\end{equation}}

\definecolor{pass}{rgb}{0,0,0.8}
\definecolor{pass1}{rgb}{0,0,0.5}

\definecolor{pass2}{rgb}{50, 50, 50}

\begin{document}
  \maketitle

    \begin{abstract}
      Time-fractional semilinear and quasilinear parabolic equations with a Caputo time derivative of order $\alpha\in(0,1)$ are
      considered, solutions of which  exhibit a singular behaviour at an initial time of type $t^\sigma$
      for any fixed $\sigma \in (0,1) \cup (1,2)$.
      The L1 scheme in time is combined with a general class of discretizations for the semilinear term.
For such discretizations, we
obtain sharp pointwise-in-time error bounds on graded temporal meshes with arbitrary degree of grading.
Both semi-discretizations in time and full discretizations using finite differences and finite elements in space are addressed.
The theoretical findings are illustrated by numerical experiments.
    \end{abstract}



\section{Introduction}\label{sec:intro}
{
Fractional-in-time semilinear and quasilinear parabolic equations will be considered.
The  semilinear problem is of the form
\beq\label{problem}
\begin{array}{l}
\pt_t^{\alpha}u+\LL u+f(u)=0\quad\mbox{for}\;\;(x,t)\in\Omega\times(0,T],\\[0.2cm]
u(x,t)=0\quad\mbox{for}\;\;(x,t)\in\pt\Omega\times(0,T],\qquad
u(x,0)=u_0(x)\quad\mbox{for}\;\;x\in\Omega,
\end{array}
\eeq
{\color{blue}{(with  a more general $f$ addressed in Remarks \ref{rem_ftu}} and \ref{rem_fxtu}).}
This problem is posed in a bounded Lipschitz domain  $\Omega\subset\R^d$ (where $d\in\{1,2,3\}$), and involves
a linear second-order elliptic  spatial operator $\LL$ defined on temporal mesh $\{t_j\}_{j\ge0}$ by
\beq\label{LL_def}
\LL u := \sum_{k=1}^d \Bigl\{-\pt_{x_k}\!(a_k(x,t)\,\pt_{x_k}\!u) + b_k(x,t)\, \pt_{x_k}\!u \Bigr\}+c(x,t)\, u,
\eeq
with sufficiently smooth coefficients $\{a_k\}$, $\{b_k\}$ and $c$ in $\bar\Omega$, for which we assume that $a_k>0$ in $\bar\Omega$,
and also both $c\ge 0$ and $c-\frac12\sum_{k=1}^d\pt_{x_k}\!b_k\ge 0$.
We shall also consider a quasilinear equation of the form $\pt_t^{\alpha}u+\QQ u+f(x,t,u)=0$,
where
$\QQ u := -\sum_{k=1}^d \pt_{x_k}\!\bigl\{a_k(x,t,u)\,\pt_{x_k}\!u + b_k(x,t,u) \bigr\}$ (see \S\ref{sec_quasi} for details).
The operator $\pt_t^\alpha$, for some $\alpha\in(0,1)$, is
the Caputo fractional derivative in time defined \cite{Diet10} by
\begin{equation}\label{CaputoEquiv}
\pt_t^{\alpha} u(\cdot,t) :=  \frac1{\Gamma(1-\alpha)} \int_{0}^t(t-s)^{-\alpha}\, \pt_s u(\cdot, s)\, ds
    \qquad\text{for }\ 0<t \le T,
\end{equation}
where $\Gamma(\cdot)$ is the Gamma function, and $\pt_s$ denotes the partial derivative in $s$.

  It will be assumed throughout the paper that the exact solution satisfies
\beq\label{ass_u_pde}
\|\partial_t^l u (\cdot, t)\|_{L_p(\Omega)}\lesssim 1+t^{\sigma-l}\qquad \mbox{for}\;\; l=0,1,2,\;\;t\in(0,T],
\eeq
with some $\sigma \in (0,1) \cup (1,2)$, for $p=2$ or $p=\infty$.
While a more typical assumption corresponds to the particular case $\sigma:=\alpha$, other values of $\sigma$
have been considered, including $\sigma=2\alpha$ for $p=\infty$ \cite{Gracia_2alpha}, and  $\sigma<\alpha$ for $p=2$ due to non-smooth inital data \cite{JLZ19}.
The singularities with $\sigma>\alpha$ are also related to delay equations, such as
a linear version of \eqref{problem} with an additional delay term $u(\cdot, t-\theta)$ ($\theta>0$)
considered in  \cite{delay_cmam},
with a typical solution satisfying
$|\partial_t^l u (\cdot, t)|\lesssim 1+(t-\theta j)^{\sigma_j-l}$ on each interval $\theta(j,j+1]$ ($j\ge0$), where $\sigma_j:=\alpha(j+1)$.

The above problems will be discretized in time using the L1 discrete fractional-derivative operator, defined by
\beq\label{delta_def}
\delta_t^{\alpha} U^m :=  \frac1{\Gamma(1-\alpha)} \sum_{j=1}^m \delta_t U^j\!\int_{t_{j-1}}^{t_j}\!\!(t_m-s)^{-\alpha}\, ds,
\qquad
\delta_t U^j:=\frac{U^j-U^{j-1}}{t_j-t_{j-1}}.\vspace{-0.1cm}%
\eeq

A wide class of discretizations in time for the semilinear term will be considered, which, in the context of
 the semilinear problem~\eqref{problem}, leads
to an L1-type semidiscretization of in time, $\forall\,m=1,\ldots,M$:
\beq\label{semi_semidiscr_method}
\delta_t^\alpha U^m +\LL U^m+ F(U^m,U^{m-1})=0\;\;\mbox{in}\;\Omega,\;U^m=0\;\;\mbox{on}\;\pt\Omega,\;\; \;U^0=u_0.
\eeq
Here $f(u)$ 
is discretized using $F(\cdot,\cdot)$, which satisfies general conditions A1 and A2 below,
formulated under the assumption that the exact solution $u$ of \eqref{problem} is bounded for $0\le t\le T$, with its range denoted by $\RR_u\subseteq \R$.
\smallskip

\begin{itemize}[leftmargin=0.7cm]
\item[{\bf A1}] (Consistency)
For some fixed positive constant $L$ and fixed $q\ge 1$, let
%
%
$$
|F(v,w)-f(v)|\le L |v-w|^q\qquad \forall\,v,\,w\in\RR_u\subseteq \R.
$$

\item[{\bf A2}] (One-sided Lipschitz condition)
Let $F(\cdot,\cdot)$ be continuous in $\R^2$ and,
for some constant $\lambda_0\ge0$, let the function
$F(v,w)+\lambda_0 v$ be non-decreasing in $v$ for any fixed $w\in\R$, or, equivalently,
\begin{align*}
F(v+\nu,w)-F(v, w)&\ge -\lambda_0\, \nu \qquad \forall\,v,\,w\in \R,\;\nu\ge 0.
\\[-3pt]
\intertext{Furthermore, let $F(\cdot,\cdot)$ satisfy the Lipschitz condition in the second argument with some constant $\lambda_1\ge 0$:\vspace{-3pt}}
\bigl|F(v, w+\omega)-F(v, w)\bigr|&\le \lambda_1|\omega| \qquad \forall\,v,\,w\in \RR_u\,,\;\omega\in \R.
\end{align*}
(
Note that
$\pt_v F(v,w)\ge-\lambda_0$ and $|\pt_w F(v,w)|\le \lambda_1$ are sufficient for A2.)
\end{itemize}
\medskip

Example 1 (Implicit scheme). One standard choice is $F(U^m,U^{m-1})=f(U^m)$, which corresponds to
$F(v,w)=f(v)$ satisfying A1 and A2 with
$L=\lambda_1=0$ and $\lambda_0=\max\{0,\,\sup_{\R}(- f')\}$.
(For  $\sigma=\alpha$, this scheme was considered in \cite{kopteva_semilin}.)
\smallskip

Example 2 (Convex splitting \cite{Du2020}).
For an Allen-Cahn model with $f(u)=u^3-u$, one choice  is
$F(U^m,U^{m-1})=(U^m)^3-U^{m-1}$, which corresponds to $F(v,w):=v^3-w$ satisfying A1 and A2 with
 $L=1$, $q=1$, $\lambda_0=0$, and $\lambda_1=1$.
\smallskip

Example 3 (First-order IMEX \cite{jin18, maskari19, Rasheed24}). A first-order IMEX scheme is given by \eqref{simplest_b} with $F(U^m,U^{m-1})=f(U^{m-1})$, which corresponds to
$F(v,w):=f(w)$ satisfying A1 and A2 with
$q=1$, $L=\lambda_1=\sup_{\R} |f'|$,
  and $\lambda_0=0$ .
\smallskip

Example 4 (Newton-iteration-type second-order IMEX \cite{HLiao1, DLi1}).
A higher-order IMEX scheme  is given by \eqref{simplest_b}
with  $F(v,w):=f(w)+[v-w]\,f'(w)$, which satisfies A1 and A2 with
 $q=2$, $L=\frac12\sup_{\R} |f''|$,
  and $\lambda_0=\max\{0,\,\sup_{\R}(- f')\}$,
  while 
  $$
  F(v, w+\omega)-F(v, w)=\bigl\{f(w+\omega)- f(w)\bigr\}-\omega\,f'(w+\omega)+[v-w]\,\bigl\{f'(w+\omega)-f'(w)\bigr\}
  $$
  yields
  $\lambda_1=2\sup_{\R}|f'|+2L\,{\rm diam}(\RR_u)$.
\smallskip

Other examples include a first-order stabilized IMEX scheme \cite{Ji20}, which
%
corresponds to  $F(U^m,U^{m-1}) = f(U^{m-1}) + S\,(U^m -U^{m-1})$
with a sufficiently large stabilization parameter $S$,
and a second-order explicit scheme with
$f(U^m)$ approximated by $f(U^{m-1}+\tau_m\delta_t U^{m-1})$ \cite{Plociniczak24}.
As the latter involves $U^{m-2}$, our analysis does not applies directly, but it easily extends to this case as well.
\smallskip

Our results 
are novel and different from the earlier literature as follows.
\smallskip
\begin{itemize}[leftmargin=0.5cm]
  \item
  Graded meshes for the singularity of type \eqref{ass_u_pde} with $\sigma=\alpha$ have been addressed in numerous works; see, e.g.,
  \cite{sorg17,liao18,NK_MC_L1,Kopteva_Meng}
    (see also \cite{Gracia_2alpha}  for $\sigma=2\alpha$, discussed in Remark~\ref{rem_earlier_sigma}).
  By contrast, the error analysis for a general $\sigma$ on graded meshes has been addressed only on uniform temporal meshes
  \cite{DLi1}. 
  We close this theoretical gap by giving sharp pointwise-in-time error bounds in the $L_2(\Omega)$ and $L_\infty(\Omega)$ norms
  for any $\sigma \in (0,1) \cup (1,2)$ and any value of the grading parameter $r\ge 1$.

  Note also that to deal with this more general case, we had to ensure that
 our key stability result, Theorem~\ref{theo_main_stab_semi}, applies to a more general range of $r\ge 1$
  (compared to a similar result in \cite{Kopteva_Meng}, as discussed in Remark~\ref{rem_extension}).
  While this may seem a minor generalization, it required structural changes in the earlier proof (to which we devote an entire \S\ref{ssec_extension}),
  and may be of independent interest.

\smallskip

  \item
 While, as discussed above, various discretizations of the semilinear term
  have been considered in the literature,
  we present general and sharp pointwise-in-time error analysis for a very general class of such discretizations
 and explicitly show how the errors depend on the solution singularity (i.e. on $\sigma$ in \eqref{ass_u_pde}), the mesh grading parameter $r$,
  and properties of $F(\cdot,\cdot)$
  (in particular, on $q$ in A1).
\smallskip

  \item
  While there is a huge literature on linear and semi-linear subdiffusion equations, there are very few papers
\cite{Jin2024,Lopez2023,Plociniczak2023}
on the numerical analysis of quasilinear subdiffusion equations. So far only uniform temporal meshes have been addressed
for self-adjoint quasilinear elliptic operators (i.e. with $b_k:=0$).
%
Since the coefficients in \eqref{problem} are functions of $x$ and $t$,
our error analysis framework very naturally extends to non-self-adjoint quasilinear operators.
Thus we extend our error bounds to quasilinear subdiffusion equations on graded temporal meshes,
which, to the best of our knowledge, have not been addressed in the literature in this context.
\end{itemize}
\smallskip

  The paper is organised as follows. A key stability result \eqref{main_stab_new}
for the discrete fractional-derivative operator is obtained in \S\ref{sec_stab}
 and applied 
 for the case without spatial derivatives in \S\ref{sec_error_ode}.
 The resulting pointwise-in-time error bounds
  are extended to semilinear and quasilinear fractional-in-time parabollic equations, respectively, in
  \S\ref{sec_semi} and \S\ref{sec_quasi}.
  Our theoretical findings are illustrated by numerical experiments in \S\ref{sec_Num}.

\smallskip

{\it Notation.}
We write
 $a\sim b$ when $a \lesssim b$ and $a \gtrsim b$, and
$a \lesssim b$ when $a \le Cb$ with a generic constant $C$ depending on $\Omega$, $T$, $u_0$,
$f$, and $\alpha$,
but not 
%
 on the total numbers of degrees of freedom in space or time.
  Also, 
  we shall use the standard 
  spaces $L_p(\Omega)$ for $1 \le p \le \infty$,
  as well as 
  the Sobolev space
  $H^1_0(\Omega)$ 
  of functions 
  vanishing on $\pt\Omega$
  and having first-order derivatives in $L_2(\Omega)$.

\section{Stability in time}\label{sec_stab}

Throughout the paper, we shall assume that the temporal mesh is quasi-graded, with a grading parameter $r\ge 1$,
in the sense that
\beq\label{t_grid_gen}
\tau := t_1\simeq M^{-r},\qquad
\tau_j:=t_j-t_{j-1}
\lesssim\tau^{1/r}t_j^{1-1/r}
\qquad 
\forall\,j=1,\ldots,M.
\eeq
One notable example of \eqref{t_grid_gen} is a standard graded grid $\{t_j=T(j/M)^r\}_{j=0}^M$.

\begin{remark}\label{rem_tau_j_1}
It is useful to note that
\eqref{t_grid_gen} immediately implies that, with a sufficiently small constant $c=c(r)\in(0,\frac12)$, one has
 $t_{j-1}\ge c t_j$ $\forall\,j\ge2$ .
Indeed, choosing $c$ sufficiently small, if $\tau/t_j\le c$, from \eqref{t_grid_gen} one gets $\tau_j/t_j\lesssim c^{1/r}\le \frac12$
or $t_{j-1}\ge\frac 12 t_j\ge c t_j$. Otherwise, i.e. if $\tau/t_j> c$, from $t_{j-1}\ge\tau$ one again gets $t_{j-1}\ge  c t_j$.
\end{remark}

The main result of this section, which will be crucial for all error bounds that we obtain below, is the following stability property of the discrete fractional operator.

\begin{theorem}[stability in time]\label{theo_main_stab_semi}
Given $\gamma\in\R$ and $\lambda_0,\,\lambda_1\ge0$,
let  the temporal mesh 
satisfy $\lambda_0\tau_j^{\alpha}< \{\Gamma(2-\alpha)\}^{-1}$ $\forall\,j\ge1$. Additionally, suppose that either one has
\eqref{t_grid_gen} with some
$r\ge 1$, 
or $\gamma\le \alpha-1$, and $\max\{\tau_j\}$ is sufficiently small.
Then 
\beq\label{main_stab_new}
\left.\!\!\!\!\begin{array}{c}
(\delta_t^\alpha-\lambda_0) U^j-\lambda_1 U^{j-1}\lesssim (\tau/ t_j)^{\gamma+1}
\\[0.0cm]
\forall j\ge1,\;\;\; U^0=0
\end{array}\hspace{-0.2cm}\right\}
 \Rightarrow
 \begin{array}{l}
U^j\lesssim
\U^j(\tau;\gamma)\\
\quad{}:=
\ell_\gamma\tau t_j^{\alpha-1}(\tau/t_j)^{\min\{0,\,\gamma\}},
\end{array}
\eeq
where $\ell_\gamma=\ell_\gamma(t_j):=1+\ln(t_j/\tau)$ for $\gamma=0$ and $\ell_\gamma:=1$ otherwise.
\end{theorem}

The above theorem is established (albeit under a certain restriction on $r$, as explained in Remark~\ref{rem_extension}) in \cite[Theorem~2.1]{Kopteva_Meng} for a significantly simpler case $\lambda_0=\lambda_1=0$,
and subsequently extended to the case $\lambda_0> 0=\lambda_1$ and $\gamma\neq 0$ in \cite[Theorem~3.1]{kopteva_semilin}.
In our proof below, for an extension  to the case $\gamma= 0$, we shall use an idea from \cite[Theorem~2.3]{NK_AML_L2nonl}.

\begin{remark}[dropping the restriction $r\le\frac{2-\alpha}{\alpha}$]\label{rem_extension}
Both \cite{Kopteva_Meng} and \cite{kopteva_semilin}
addressed the case $\sigma=\alpha$, for which it was sufficient to assume in \cite[Theorem~2.1]{Kopteva_Meng} and \cite[Theorem~3.1]{kopteva_semilin}
that either $1\le r\le \frac{2-\alpha}{\alpha}$ or $\gamma\le \alpha-1$.
However, an inspection of our error analysis in section~\ref{sec_error_ode} below shows that for the case $\sigma<\alpha$, we need to extend \cite[Theorem~2.1]{Kopteva_Meng}
to arbitrarily large values of $r\ge 1$. This requires non-trivial modifications in the proof of \cite[Theorem~2.1]{Kopteva_Meng}, so we devote an entire section~\ref{ssec_extension}
to this generalization.
Hence, in the proof of Theorem~\ref{theo_main_stab_semi}, we shall be able to assume that \eqref{main_stab_new} is valid for $\lambda_0=\lambda_1=0$.
\end{remark}

The proof of Theorem~\ref{theo_main_stab_semi} hinges on the following comparison principle.

\begin{lemma}[comparison principle for the operator in \eqref{main_stab_new}]\label{lem_comparison}
Given $\lambda_0,\,\lambda_1\ge0$, let the temporal mesh satisfy $\lambda_0\tau_j^{\alpha}<\{\Gamma(2-\alpha)\}^{-1}$ $\forall\,j\ge1$.
Then $V^0\le 0$ and  $(\delta_t^\alpha-\lambda_0)V^m-\lambda_1 V^{m-1}\le 0$ $\forall\,m\ge1$
imply $V^m\le 0$ $\forall\,m\ge0$.
\end{lemma}

\begin{proof}
The proof is by induction.
First,
a straightforward calculation (see, e.g., \cite{NK_MC_L1}) shows that
\eqref{delta_def} on an arbitrary temporal mesh $\{t_j\}$ can be represented as
\beq\label{delta_def_kappa}
\delta^\alpha_t V^m=\kappa_{m,m} V^m-\sum_{j=0}^{m-1}\kappa_{m,j}V^j,
\quad
\kappa_{m,m}=\frac{\tau_m^{-\alpha}}{\Gamma(2-\alpha)}, \;\; \kappa_{m,j}>0\;\;\forall\,m\ge j.
\eeq
Now, assuming that $V^j\le 0$ $\forall\,j<m$, one can rewrite $(\delta_t^\alpha-\lambda_0)V^m-\lambda_1 V^{m-1}\le 0$
as
$$
(\kappa_{m,m}-\lambda_0)V^m\le \lambda_1 V^{m-1}+\sum_{j<m}\kappa_{m,j}V^j,
$$
which immediately yields the desired $V^m\ge 0$.
\end{proof}

\begin{lemma}[{\cite[Lemma~3.2]{kopteva_semilin}}]\label{lem_axu_barrier}
Given any $\lambda\ge 0$, any fixed constant $0<c_0<\frac12\{\lambda \Gamma(2-\alpha)\}^{-1/\alpha}$,
and any fixed mesh point $0\le t_m\le T$, if  $\tau_j\le\frac12 c_0$ $\forall\,j\ge 1$, then
there exists a non-decreasing discrete function
$\{B^j\}_{j=0}^M$ such that $0\le B^j\lesssim 1$ $\forall\,j\ge1$ and
\beq\label{aux_B}
B^j=0\;\;\forall\,j\le m,\quad\;\;
(\delta_t^\alpha- \lambda)\, B^j \ge Q^j:=\!\left\{\begin{array}{ll}
0\,&\mbox{for~}t_j< t_m+c_0\\
1\,&\mbox{for~}t_j\ge t_m+c_0
\end{array}\right.
\;\forall\,j\ge1.
\eeq
\end{lemma}

\begin{proof}[Proof of Theorem~\ref{theo_main_stab_semi}]
%
%
For the particular case
$\lambda_0=\lambda_1=0$
(whether one has \eqref{t_grid_gen} with $r\ge 1$, 
or $\gamma\le \alpha-1$), the desired stability result \eqref{main_stab_new} is proved in
Lemma~\ref{lem_stability1}
(which is a generalization of \cite[Theorem~2.1]{Kopteva_Meng}), and implies that
\beq\label{main_stab_B}
\left.\begin{array}{c}
\delta_t^\alpha {\B}_\gamma^j= (\tau/ t_j)^{1+\gamma}
\\[0.2cm]
\forall j\ge1,\;\;\; {\B}_\gamma^0=0
\end{array}\right\}
\Rightarrow
0\le{\B}_\gamma^j\lesssim {\mathcal U}_\gamma^j=
\ell_\gamma(t_j)\,\tau^\alpha
(\tau/t_j)^{1+\min\{0,\,\gamma\}-\alpha}
\;\;\;\forall j\ge 1,
\eeq
where we use a slightly more compact notation ${\mathcal U}_\gamma^j$ for ${\mathcal U}^j(\tau\,;\gamma)$ from \eqref{main_stab_new}, and also rewrite this quantity in
an different equivalent form.
To extend \eqref{main_stab_B} to our more general case, set
$$
\lambda:=\lambda_0+\lambda_1\ge 0,\qquad
\gamma^*:=\min\{0,\,\gamma\}-\alpha<0,\qquad
\bar\ell_\gamma:=\max_{j\le M}\ell_\gamma(t_j)=\ell_\gamma(T).
$$
Recall Remark~\ref{rem_tau_j_1} for $j\ge2$ and $\B_\gamma^0=0$ for $j=1$,
which yield ${\B}_\gamma^{j-1}\lesssim {\B}_\gamma^{j}$, and hence, with some constant $C_\gamma>0$,
$$
\lambda_0 {\B}_\gamma^j+\lambda_1 {\B}_\gamma^{j-1}\le \lambda C_\gamma\, {\mathcal U}_\gamma^j=\lambda C_\gamma\, \bar\ell_\gamma\,\tau^\alpha
(\tau/t_j)^{1+\gamma^*}.
$$
Hence,
\begin{align}\label{B_gamma}
(\delta_t^\alpha-\lambda_0){\B}_\gamma^j-\lambda_1 {\B}_\gamma^{j-1}&\ge (\tau/ t_j)^{1+\gamma}-\lambda C_\gamma\, \bar\ell_\gamma \tau^\alpha (\tau/t_j)^{1+\gamma^*},
\\
\intertext{and, similarly, only noting that $\bar\ell_{\gamma^*}=1$ (as $\gamma^*\neq 0$)
and
$\min\{0,\,\gamma^*\}-\alpha=\gamma^*-\alpha$, one gets}
(\delta_t^\alpha-\lambda_0){\B}_{\gamma^*}^j-\lambda_1 {\B}_{\gamma^*}^{j-1}&\ge (\tau/ t_j)^{1+\gamma^*}-\lambda C_{\gamma^*} \tau^\alpha (\tau/t_j)^{1+\gamma^*-\alpha}
\label{B_star}\\{}\notag
&=\textstyle
\frac12 (\tau/ t_j)^{1+\gamma^*}+\frac12 (\tau/ t_j)^{1+\gamma^*}\bigl[1-2 \lambda C_{\gamma^*} t_j^\alpha \bigr].
\end{align}

Next, using the above ${\B}_{\gamma}^j$ and ${\B}_{\gamma^*}^j$, as well as $B^j$ from Lemma~\ref{lem_axu_barrier} (assuming that the conditions of the latter lemma are satisfied), set
\beq\label{W_def}
{\mathcal W}^j:={\B}_\gamma^j+2\lambda C_\gamma\, \bar\ell_\gamma \tau^\alpha \Bigl\{ {\B}_{\gamma^*}^j+C_*\tau^{1+\gamma^*} B^j\Bigr\}
\eeq
with a sufficiently large constant $C_*$.
Now, combining the above observations \eqref{B_gamma} and \eqref{B_star} with \eqref{aux_B}, one gets
\begin{align}\label{delta_W_aux}
(\delta_t^\alpha-\lambda_0){\mathcal W}^j&-\lambda_1{\mathcal W}^{j-1}\ge  (\tau/ t_j)^{1+\gamma}
\\{}\notag
&+
\lambda C_\gamma\, \bar\ell_\gamma \tau^\alpha \Bigl\{(\tau/ t_j)^{1+\gamma^*}
\bigl[1-2 \lambda C_{\gamma^*} t_j^\alpha\bigr] + 2 C_*\tau^{1+\gamma^*} Q^j\Bigr\}.
\end{align}
Importantly, in the statement of Lemma~\ref{lem_axu_barrier}, the discrete function $B^j$ is non-decreasing in $j$, so
 $\lambda_0 B^j+\lambda_1 B^{j-1}\ge \lambda\, B^j$.
 Hence, the term $(\delta_t^\alpha- \lambda)\, B^j$ in \eqref{aux_B} can be replaced by $(\delta_t^\alpha- \lambda_0)\, B^j-\lambda_1 B^{j-1}$, which we used in the above
 evaluation.

Note that \eqref{delta_W_aux} immediately yields
\beq\label{delta_W}
(\delta_t^\alpha-\lambda_0){\mathcal W}^j-\lambda_1{\mathcal W}^{j-1}\ge
(\tau/ t_j)^{1+\gamma}
\eeq
whenever $[1-2 \lambda C_{\gamma^*} t_j^\alpha]\ge 0$.
Otherwise, i.e. if $[1-2 \lambda C_{\gamma^*} t_j^\alpha]< 0$, we shall
obtain \eqref{delta_W}
by ensuring that then $Q^j=1$ and choosing $C_*$ sufficiently large
(so that $2 C_*\tau^{1+\gamma^*} Q^j=2 C_*\tau^{1+\gamma^*}$ dominates
$(\tau/ t_j)^{1+\gamma^*}
\,2 \lambda C_{\gamma^*} t_j^\alpha$).
To ensure $Q^j=1$ whenever $[1-2 \lambda C_{\gamma^*} t_j^\alpha]< 0$,
in Lemma~\ref{lem_axu_barrier} choose a sufficiently small constant $c_0$ satisfying $2 \lambda C_{\gamma^*} (2c_0)^\alpha\le 1$, and the maximal $t_m\le c_0$ (so whenever $[1-2 \lambda C_{\gamma^*} t_j^\alpha]<0$, one gets $t_j\ge t_m+c_0$).
Note that Lemma~\ref{lem_axu_barrier} requires
$\tau_j\le \frac12 c_0$ $\forall\,j\ge 1$.
(If this is not the case for some $j$, then \eqref{t_grid_gen} implies $\tau^{1/r}\gtrsim t_j^{-(1-1/r)}\gtrsim 1$, so $M\simeq 1$ $\forall\,j\ge 1$, so the desired result
becomes straightforward.)

Thus, we have obtained \eqref{delta_W} $\forall\, j\ge1$.
Hence, in view of the comparison principle of Lemma~\ref{lem_comparison},
for $U^j$ satisfying the assumptions in \eqref{main_stab_new} one concludes
that $U^j\lesssim {\mathcal W}^j$.
Finally note that \eqref{W_def} implies
$$
U^j\lesssim{\mathcal W}^j\lesssim
\ell_\gamma(t_j)\,\tau^\alpha
(\tau/t_j)^{1+\gamma^*}
+
\bar\ell_\gamma \tau^\alpha \Bigl\{\tau^\alpha
(\tau/t_j)^{1+\gamma^*-\alpha}+\tau^{1+\gamma^*} B^j\Bigr\},
$$
while the desired assertion is equivalent to $U^j\lesssim \ell_\gamma(t_j)\,\tau^\alpha
(\tau/t_j)^{1+\gamma^*}$.
To show that the latter is indeed true, first, note that one has $B^j\neq 0$ for $t_j\ge t_m\gtrsim 1$, so the term $\tau^{1+\gamma^*} B^j$ enjoys the same bound as $\tau^\alpha
(\tau/t_j)^{1+\gamma^*-\alpha}$.
Combining this observation with
$$
\bar\ell_\gamma\,t_j^\alpha\lesssim \bigl[1+\ln(t_j/\tau)+\ln(T/t_j)\bigr]t_j^\alpha\lesssim 1+\ln(t_j/\tau)=\ell_\gamma(t_j)
$$
yields the desired bound $U^j\lesssim{\mathcal W}^j\lesssim \ell_\gamma(t_j)\,\tau^\alpha
(\tau/t_j)^{1+\gamma^*}$.
\end{proof}


\subsection{Generalization of \cite[Theorem~2.1(i)]{Kopteva_Meng} for any $r\ge 1$ (dropping the restriction $r\le\frac{2-\alpha}{\alpha}$)}\label{ssec_extension}
Our task is to establish \eqref{main_stab_new} for the case $\lambda_0=\lambda_1=0$.
This result was proved in \cite[Theorem~2.1(i)]{Kopteva_Meng} under the assumption $1\le r\le \frac{2-\alpha}\alpha$.
However, for the purpose in this paper (see Remark~\ref{rem_extension}) we need to relax this assumption.
This is addressed by the following result \eqref{B_def}, valid for any $r\ge 1$, which, while not being identical  with  \eqref{main_stab_new} for the case $\lambda_0=\lambda_1=0$,
is equivalent to the latter (in view of $\delta_t^\alpha$ being a linear operator).

\begin{remark}\label{rem_nontrivial}
Note that the proof of \cite[Theorem~2.1(i)]{Kopteva_Meng} hinges on \cite[Lemma~2.3]{Kopteva_Meng}, while an inspection of the proof of the latter shows that
it cannot be extended to the more general $r\ge 1$.
The key issue lies in that the truncation error $|(\pt_t^\alpha-\delta_t^\alpha) B|$ for a single barrier function $B$ of type \eqref{B_L1_def}
is not dominated by $\frac12\pt_t^\alpha B$  unless $r\le \frac{2-\alpha}\alpha$.
To rectify this, we shall now present a new proof, which builds on the ideas used in \cite[Lemmas~2.3, 2.5, and 2.6]{Kopteva_Meng}
in a non-trivial way,
the key  being that for appropriately chosen linear combinations $W_N$ of functions of type \eqref{B_L1_def} (see~\eqref{W_N_def}),
we instead estimate the ensemble truncation error $(\pt_t^\alpha-\delta_t^\alpha) W_N$, which, as we show below, is dominated by $\frac12\pt_t^\alpha W_N$.
\end{remark}

\begin{lemma}\label{lem_stability1}
Given $\gamma\in\R$,
let the temporal mesh satisfy
 \eqref{t_grid_gen} 
with some $r\ge 1$.
Then\vspace{-0.2cm}
\beq\label{B_def}
\left.\!\!\!\!\begin{array}{c}
\delta_t^\alpha V^m\le t_m^{-1}(\tau/ t_m)^{\gamma}
\\[0.2cm]
\forall m\ge1,\;\;\; V^0=0
\end{array}\hspace{-0cm}\right\}
\; \Rightarrow\;
V^m\lesssim
\tau^{-1}\U^m(\tau;\gamma)
=
\ell_\gamma t_m^{\alpha-1}(\tau/t_m)^{\min\{0,\,\gamma\}},
\eeq
where $\ell_\gamma=\ell_\gamma(t_m)=1+\ln(t_m/\tau)$ for $\gamma=0$ and $\ell_\gamma:=1$ otherwise
(and the constants contained in $\lesssim$ depend only on $\alpha$, $r$, and $\gamma$).
\end{lemma}\vspace{-0.3cm}

\begin{proof}
To simplify the presentation, we shall give a proof for a standard graded grid $\{t_j=T(j/M)^r\}_{j=0}^M$, while one can
extend it to a more general
\eqref{t_grid_gen} along the lines of the proof of \cite[Theorem~2.1(i)]{Kopteva_Meng}.

Fix $\alpha$ and $r$,  and note that it suffices to prove the desired result \eqref{B_def} for
\beq\label{gamma_star}
\gamma\le\gamma^*:=\min\{\alpha,\,(2-\alpha)/r\}.
\eeq
Indeed, 
for any $\gamma>\gamma^*$, the discrete maximum/comparison principle for the operator $\delta_t^\alpha$
implies that if $\delta_t^\alpha \bar V^m= t_m^{-1}(\tau/ t_m)^{\gamma^*}$ subject to $\bar V^0=0$, then
$V^m\le \bar V^m$ $\forall\, m\ge 0$.
 Assuming \eqref{B_def} is established for $\gamma=\gamma^*$, one concludes that $V^m\le \bar V^m\lesssim t_m^{\alpha-1}$, which is equivalent to the assertion in \eqref{B_def}
 for $\gamma>\gamma^*$ under consideration.
 Thus, \eqref{gamma_star} will be assumed for the remainder of the proof.

Set 
\beq\label{B_L1_def}
\beta:=1-\alpha,\qquad B(s\,;t_n):=\min\bigl\{(s/t_n)\,t_n^{-\beta}, s^{-\beta} \bigr\},
\eeq
where $n \ge 1$.
It is shown in the proof of \cite[Lemma~2.3 (see (2.4))]{Kopteva_Meng} that
\beq\label{app_B_eq}
\Gamma(\beta)\,
\pt^\alpha_t B(t\,;t_n)
\ge\left\{
\begin{array}{lll}
t_n^{-\beta} t^{-\alpha}\,(\beta t_n/t)
&=
\beta\, (t_n/t)^{\alpha}\,t^{-1}
&\quad\forall\,t>t_n>0,\\
\beta^{-1}\, (t/ t_n)^{\beta}\,t_n^{-1}
&\ge 0
&\quad\forall\,0<t\le t_n.
\end{array}\right.
\eeq
%
Next, for any fixed $0\le N\le \infty$, consider the barrier function
\beq\label{W_N_def}
W_N(s):=\sum_{n=0}^N  c_nB(s\,;t_{p_n}),\qquad\mbox{where}\quad p_n:= 2^n p\qquad
c_n:=(t_p/t_{p_n})^\gamma,
\eeq
and $1\le p\lesssim 1$ is a sufficiently large fixed integer to be chosen later in the proof
to ensure~\eqref{bar_B_new}, so the dependence on $p$ will be shown explicitly in all bounds of type $\lesssim$ leading to~\eqref{bar_B_new}.
Note that (in view of $t_j=T(j/M)^r$)
\beq\label{c_aux2}
p_n< m\le p_{n+1}\;\;\Rightarrow\;\;
t_{p_n}/  t_m \ge t_{p_n}/  t_{p_{n+1}} \ge 2^{-r}
\;\;\mbox{and}\;\;
c_n \simeq (t_p/t_m)^{\gamma}.
\eeq
Furthermore, noting that the weights $c_n=(t_p/t_{p_n})^\gamma=2^{-\gamma r n}$ form a geometric sequence, for any $N\ge 0$, one gets
\beq\label{c_aux1}
\sum_{n=0}^N c_n\lesssim S_N:=\left\{\begin{array}{ll}
1&\mbox{if~~}\gamma>0,
\\
1+\ln(t_{p_N}/t_p)&\mbox{if~~}\gamma=0,
\\
(t_p/t_{p_N})^\gamma&\mbox{if~~}\gamma<0.
\end{array}\right.
\eeq
In particular, for $\gamma=0$, we combined  $\sum_{n=0}^N c_n=N+1$ with
$\ln(t_{p_N}/t_p)=N\ln(2^{r})$, while
for $\gamma<0$, we used $\sum_{n=0}^N c_n\lesssim c_N$.

We claim 
that one can choose a sufficiently large $1\le p\lesssim 1$ such that, for any fixed
$N\ge 0$,
\beq\label{bar_B_new}
W_N(t_m)\lesssim t_m^{-\beta}\,S_N\quad\mbox{and}\quad\delta_t^\alpha W_N(t_m)\gtrsim t_m^{-1}\,(t_p/ t_m)^{\gamma}\qquad \forall\, m\le p_{N+1},
\eeq
where the error constant in the second bound is independent of $N$,
but depends on 
$p$.

Once \eqref{bar_B_new} is established,
a comparison of $\delta^\alpha_t W_N(t_m)$ with $\delta^\alpha_t V^m$ in \eqref{B_def}
 immediately implies (in view of the discrete maximum principle for the operator $\delta_t^\alpha$
on the mesh $\{t_j,j\le  p_{N+1}\}$, and the observation that
$t_p=\tau p^r$, where $1\le p\lesssim 1$) that
$V^m\lesssim W_N(t_m)\lesssim t_m^{\alpha-1}S_N$ $\forall\,m\le p_{N+1}$.
Combining this with the  bound \eqref{c_aux1} on $S_N$ yields a bound on $V^m$ $\forall\,m\le p_{N+1}$,
which involves $(t_{p_N}/t_p)$.
The latter bound immediately implies the desired bound \eqref{B_def}
for $p_N< m\le p_{N+1}$, as for this range $t_{p_N}/t_p\le t_m/t_p\le t_m/\tau$.
As this argument applies for any fixed $N\ge 0$
(since the error constant in \eqref{bar_B_new} is independent of $N$), one gets the desired bound on $V^m$ $\forall\,p<m\le M$.
For the remaining range $ m\le p$,
one has $V^m\lesssim W_0(t_m)\lesssim t_m^{\alpha-1}S_0=t_m^{\alpha-1}$ (as $S_0=1$), which, combined with $t_m/\tau\ge 1$, again yields \eqref{bar_B_new}.
Note also that
for $\gamma>0$, one gets $S_\infty=1$, so one can simplify the above argument and use $V^m\lesssim W_\infty(t_m)\lesssim t_m^{\alpha-1}S_N$ $\forall\,m\ge 1$.

Thus, it remains to establish \eqref{bar_B_new} for any fixed $\gamma\le \gamma^*$.
The first bound in \eqref{bar_B_new} is straightforward, in view of the definition of $W_N$ combined with $B(s\,;t_{p_n)}\le s^{-\beta}$ $\forall\, n$, and \eqref{c_aux1}.
For the second bound in \eqref{bar_B_new}, we shall start with a desired bound on $\pt_t^\alpha W_N(t_m)$, and then show that
the truncation error $|(\pt_t^\alpha-\delta_t^\alpha) W_N(t_m)|$ is dominated by $\frac12\pt_t^\alpha W_N(t_m)$.
To proceed with this plan,
recall that $\pt_t^\alpha  B(t\,;t_{p_n})\ge 0$ $\forall\,t>0$.
Hence, for each $0\le n\le N$
for the range $p_n<m\le p_{n+1}$, one has
\begin{align}\label{pt_W}
\pt_t^\alpha W_N(t_m)\ge c_n\,\pt_t^\alpha  B(t_m\,;t_{p_n})&\ge c_n\{\Gamma(\beta)\}^{-1}\beta\, (t_{p_{n}}/t_m)^{\alpha}\,t_m^{-1}
\gtrsim (t_p/t_m)^{\gamma}\,t_m^{-1},
\end{align}
where we first used the first line in \eqref{app_B_eq}, and then \eqref{c_aux2}.
Similarly, for $1\le m\le p$, one gets $\pt_t^\alpha W_N(t_m)\ge c_0\,\pt_t^\alpha  B(t_m\,;t_{p})$, where $c_0=1$,
so
the second line in \eqref{app_B_eq} yields
$\pt_t^\alpha W_N(t_m)\gtrsim (t_m/t_p)^{\beta}\,t_p^{-1}\ge  p^{-r\,\max\{0,\,\beta+\gamma+1\}}\,t_m^{-1}(t_p/t_m)^\gamma$, where we also used
$t_m/t_p\ge t_1/t_p=p^{-r}$.

Our next task is to bound $|(\pt_t^\alpha-\delta_t^\alpha) W_N(t_m)|$.
When evaluating the truncation errors for each $B(t_m\,;t_{p_n})$
appearing in $W_N(t_m)$, note that while this function remains linear in $t_m$ (i.e. for $m\le{p_n}$), it does not induce any truncation error.
Hence, we immediately get $(\pt_t^\alpha-\delta_t^\alpha)W_N(t_m)=0$, and so
the second bound in \eqref{bar_B_new}, for $m\le p$.

To estimate $|(\pt_t^\alpha-\delta_t^\alpha) W_N(t_m)|$
for $m>p$, note that the evaluation in the proof of \cite[Lemma~2.3]{Kopteva_Meng} (see, in particular, \cite[(2.5)]{Kopteva_Meng}) yields, for $m>p_n$,
\begin{align*}
\Gamma(\beta)\,(\pt_t^\alpha-\delta_t^\alpha)B(t_m\,;t_{p_n})&=
\int_{t_{p_n}}^{t_m}\!\!\pt_s(B-B^I)(s\,;t_{p_n})\,(t_m-s)^{-\alpha}ds\\
&=\alpha\int_{t_{p_n}}^{t_m}\!\! (B^I-B)(s\,;t_{p_n})\,(t_m-s)^{-\alpha-1} ds.
\end{align*}
Here $B^I(s\,;t_{p_n})$ is the standard piecewise-linear interpolant of $B(s\,;t_{p_n})$ on $\{t_j\}$, and, recalling the definition of the latter,
note that
$B(s\,;t_{p_n})=B(s\,;t_{p})$ for $s>t_{p_n}$,
so
$(B^I-B)(s\,;t_{p_n})=(B^I-B)(s\,;t_{p})\ge 0$ for $s>t_{p_n}$, while $(B^I-B)(s\,;t_{p_n})=0$ for $s\le t_{p_n}$.
Combining the above observations, one concludes that, $\forall\,m> p$,
\begin{align*}
\Gamma(\beta)\,(\pt_t^\alpha-\delta_t^\alpha)W_N(t_m)
&=\sum_{n\,:\, p_n<m}c_n\,\alpha\int_{t_{p_n}}^{t_m}\!\! (B^I-B)(s\,;t_{p})\,(t_m-s)^{-\alpha-1} ds\\
&\le 
\underbrace{\Bigl\{\sum_{n\,:\, p_n<m}c_n\Bigr\}}_{{}\lesssim S_J}\,
\alpha\int_{t_{p}}^{t_m}\!\! (B^I-B)(s\,;t_{p})\,(t_m-s)^{-\alpha-1} ds.
\end{align*}
Here
$J=J(m)$ is the maximum value in the set $\{n: p_n<m\}$, i.e.
$p_J<m\le p_{J+1}$, so $\sum_{n: p_n<m} c_n\lesssim S_J$, in view of \eqref{c_aux1}.

The above integral is (up to a constant factor) the truncation error of $B(s\,;t_{p})$ on the interval $(t_p,t_m)$, which is evaluated in
\cite[(2.6) and (2.7)]{Kopteva_Meng}. In particular, the first line of \cite[(2.7)]{Kopteva_Meng} yields the following bound
(while the second line in \cite[(2.7)]{Kopteva_Meng} already relies on the restriction $r\le\frac{2-\alpha}\alpha$, which we avoid in this proof):
\beq
\label{trunc_er_B}
\left|(\pt_t^\alpha-\delta_t^\alpha)W_N(t_m)\right|
%
\lesssim S_J\,
t_m^{-1}
\Bigl[ (\tau/t_p)^{2/r}(t_p/t_m)^\alpha + (\tau/t_m)^{(2-\alpha)/r}\Bigr].
\eeq
{Alternatively, this truncation error bound can be viewed as a particular case
of a more general bound obtained in the proof of Lemma~\ref{lem_r} below (as elaborated in Remark~\ref{rem_trunc}).
Next, from \eqref{trunc_er_B} one gets}
\begin{align}
\notag
\left|(\pt_t^\alpha-\delta_t^\alpha)W_N(t_m)\right|
&\lesssim
S_J\,t_m^{-1}\,(\tau/t_p)^{(2-\alpha)/r} \Bigl[ (\tau/t_p)^{\alpha/r}(t_p/t_m)^\alpha+ (t_p/t_m)^{(2-\alpha)/r}\Bigr]
\\
\notag
&\lesssim
S_J\,t_m^{-1}\,p^{-(2-\alpha)} \Bigl[ (t_p/t_m)^{\gamma^*}\Bigr],
\end{align}
where we used
$(\tau/t_p)^{(2-\alpha)/r}=(p^{-r})^{(2-\alpha)/r}=p^{-(2-\alpha)}$, as well as
$(\tau/t_p)^{\alpha/r}\le 1$, and then $t_p/t_m<1$ combined with \eqref{gamma_star}.

A comparison of the above bound with \eqref{pt_W} (in both of which, importantly, the dependence on $p$
is shown explicitly) implies that it remains
 to prove that
 \beq\label{s_J}
S_J\,\Bigl[ (t_p/t_m)^{\gamma^*}\Bigr]\lesssim (t_p/t_m)^\gamma\qquad \mbox{for any}\quad \gamma\le\gamma^*.
  \eeq
 Then, indeed, 
 choosing a sufficiently large $p$ ensures that $|(\pt_t^\alpha-\delta_t^\alpha)W_N(t_m)|$
 is dominated by $\frac12\pt_t^\alpha W_N(t_m)$ for any $m\ge p$, which yields the second bound in \eqref{bar_B_new}.
For $0<\gamma\le \gamma^*$, recalling \eqref{c_aux1}, one gets $S_J=1$, so \eqref{s_J} is obvious.
For $\gamma\le 0$, recall also that
$J=J(m)$ is such that
$p_J<m\le p_{J+1}$, so $t_{p_J}/t_p\le t_m/t_p$, which, combined with \eqref{c_aux1} again yields \eqref{s_J}
In particular, for $\gamma=0$, we used $S_J\le 1+\ln(t_m/t_p)$ leading to
$ S_J(t_p/t_m)^{\gamma^*} \lesssim 1$,
while for any $\gamma<0$, we used $S_J\le (t_p/t_m)^{\gamma}$ combined with
$ (t_p/t_m)^{\gamma^*}\le 1$.
This completes the proof of \eqref{bar_B_new}, and, hence, of \eqref{B_def}.
\end{proof}

\section{Error for a paradigm example without spatial derivatives}\label{sec_error_ode}
We shall first illustrate our error analysis using the   simplest semilinear example without spatial derivatives and its disretization of type~\eqref{semi_semidiscr_method}:
\begin{subequations}\label{simplest}
\begin{align}
\label{simplest_a}
\pt_t^\alpha u(t)\;+\;\;f(u)\;\;&=g(t)&&\hspace{-0cm}\mbox{for}\;\;t\in(0,T],&&\hspace{-0.0cm} u(0)=u_0,
\\[3pt]
\delta_t^\alpha U^m+F(U^m,U^{m-1})&=g(t_m)&&\hspace{-0cm}\mbox{for}\;\;m=1,\ldots,M,&&\hspace{-0.0cm} U^0=u_0.
\label{simplest_b}
\end{align}
\end{subequations}
It will be assumed throughout this section that \eqref{simplest_a} has a unique solution that, for some $\sigma \in (0,1) \cup (1,2)$, satisfies
\beq\label{ass_u_ode}
\textstyle\bigl|\frac{d^l}{dt^l}u(t)\bigr|\lesssim 1+t^{\sigma-l}\qquad \mbox{for}\;\; l=0,1,2,\;\;t\in(0,T].
\eeq

\noindent
{\it Example.}\,
The simplest case of \eqref{simplest_a} subject to \eqref{ass_u_ode} is
with
 $f(u):=0$ and $g(t)=c_0\,t^{\sigma-\alpha}$,
 where $c_0=\Gamma(\sigma+1)/\Gamma(\sigma+1-\alpha)$,
 which yields
 $u(t)=u_0+t^\sigma$.

\subsection{Truncation error of the L1 fractional-derivative operator}

For the truncation error associated with our discrete fractional-derivative operator $\delta_t^\alpha$, we have the following bound.


\begin{lemma}[truncation error for $\delta_t^\alpha$]\label{lem_r}
Let $u$ satisfy \eqref{ass_u_ode} with some $\sigma \in (0,1) \cup (1,2)$.
Then for the truncation error $r^m:=\delta_{t}^{\alpha} u(t_m)-\pt_t^\alpha u(t_m)$
on the temporal grid~\eqref{t_grid_gen}, one has
\beq\label{r_bound_simple}
|r^m|\lesssim  \tau^{\sigma - \alpha} (\tau/t_m)^{\gamma+1}\quad\forall\,m\ge 1,
\quad\mbox{where}\;\;
\textstyle
\gamma:=\min\bigl\{\alpha,\,\,\frac{2-\alpha}{r}+\alpha-\sigma-1\bigr\}.
\eeq
\end{lemma}

\begin{proof}
We start by representing $r^m$ via the function $\chi := u - u^I$, where $u^I(t)$ is the piecewise linear interpolant of $u(t)$ on our temporal mesh $\{t_j\}$.
Note that $\forall\,s\in(t_{j-1},t_j)$ one has $\frac{u(t_{j})-u(t_{j-1})}{\tau_j}-u'(s)=-\chi'(s)$. Hence, combining \eqref{CaputoEquiv} and \eqref{delta_def}
and then integrating by parts yields
\beq\label{trun_int}
\Gamma(1-\alpha)\,r^m
=\int_0^{t_m}\!\!\!(t_m-s)^{-\alpha}\,[-\chi'(s)]ds
=\alpha
\int_0^{t_m}
\!\!\!(t_m-s)^{-\alpha-1}\chi(s)\,ds.
\eeq

Next, note that
for $s\in(t_{j-1},t_j)$ with $2\le j\le m$ (i.e. away from the initial singularity),
 in view of~\eqref{ass_u_ode},
one gets
$|\chi(s)|\lesssim (t_j-s)(s-t_{j-1})\,t_{j-1}^{\sigma-2}$. So,
using $t_{j-1}\simeq s \simeq t_j$ (as $j\ge 2$) and
$\tau_j\lesssim\tau^{1/r}s^{1-1/r}$,
 both of which follow from \eqref{t_grid_gen} (as discussed in Remark~\ref{rem_tau_j_1}) ,
yields
\beq\label{chi1}
|\chi(s)|
\lesssim
\min \bigl\{\tau_j^2,\,(t_m-s)\,\tau_j\bigr\}\, s^{\sigma-2}
\lesssim \min\bigl\{
\tau^{2/r}s^{\sigma-2/r},\,
(t_m-s)\, \tau^{1/r}s^{\sigma-1-1/r}
\bigr\}.
\eeq

For $s\in(0,t_1)$, we shall instead use
\beq\label{chi0}
|\chi(s)|\lesssim (t_1-s)\,(s^{\sigma-1}+t_1^{\sigma-1}),
\eeq
which
holds true whether $\sigma$ is in $(0,1)$ or $(1,2)$. 
For $\sigma\in(0,1)$, 
this bound follows from
$|\chi(s)|\le |u(t_1)-u(s)|+|u(t_1)-u^I(s)|$
combined with
$|u'(s)|\lesssim s^{\sigma-1}$.
To be more precise, we also employed
$|u(t_1)-u(s)|\le (t_1-s)\,\sup_{(s,t_1)}|u'|$,  as well as  $|u(t_1)-u^I(s)|=(t_1-s)\,t_1^{-1}|u(t_1)-u(0)|$, where
$|u(t_1)-u(0)|\lesssim |\int_0^{t_1}s^{\sigma-1}ds|\simeq t_1^\sigma$.
For $\sigma\in(1,2)$, we only have
$|u'(s)|\lesssim 1+s^{\sigma-1}\simeq 1$.
So we modify the argument used for $\sigma\in(0,1)$ by replacing $u(t)$ with $\tilde u(t):=u(t)-u(0)-u'(0)\,t$, which enjoys
$|\tilde u'(t)|\le\int_0^t|u''(s)|ds\lesssim t^{\sigma-1}$ (so $\sup_{(s,t_1)}|\tilde u'|\lesssim t_1^{\sigma-1}$),
while $\chi=u-u^I=\tilde u-\tilde u^I$.


Now, splitting the integral in \eqref{trun_int} into three parts, over $(0,t_1)$, $(t_1,t_m^*)$, and $(t_m^*,t_m)$,
and then using \eqref{chi0} on the first interval and \eqref{chi1} elsewhere,
one gets
\begin{align*}
|r^m |{}\lesssim
\mathcal{I}_1{}&+\mathcal{I}_2+\mathcal{I}_3:=
\int_{0}^{t_1}\!\!(t_m -s)^{-\alpha -1}(t_1-s)\,(s^{\sigma-1}+t_1^{\sigma-1})\,ds
\\
&{}+
\int_{t_1}^{t_m^*}\!\!
(t_m -s)^{-\alpha -1}\tau^{2/r}s^{\sigma-2/r}\,ds
+
\int^{t_{m}}_{t_m^*}\!\!(t_m -s)^{-\alpha}\,\tau^{1/r}s^{\sigma-1-1/r}\,ds
.
\end{align*}
Here 
$t_m^*:=\max\{t_1,\,t_m(1-\theta)\}$, where $\theta=\theta_m:=\frac12(\tau/t_m)^{1/r}\le\frac12$
(note that the second relation in \eqref{t_grid_gen} can then be rewritten as $\tau_m/t_m\lesssim \theta_m$).

To estimate $\mathcal{I}_1$,  note that $t_m/t_1\le(t_m-s)/(t_1-s)$
implies that
$(t_m -s)^{-\alpha -1}\le (\tau/t_m)^{\alpha +1}(t_1 -s)^{-\alpha -1}$.
So, with the change of variable $\hat s:=s/t_1=s/\tau$, one gets
\begin{align*}
\mathcal{I}_1&{}\le (\tau/t_m)^{\alpha +1}\!\int_{0}^{t_1}\!\!(t_1 -s)^{-\alpha}\,(s^{\sigma-1}+t_1^{\sigma-1})\,ds
\\
&{}=\tau^{\sigma-\alpha}(\tau/t_m)^{\alpha +1}
\!\int_{0}^{1}\!\!(1 -\hat s)^{-\alpha}\,(\hat s^{\sigma-1}+1)\, d\hat s
\\[2pt]
&{}\lesssim  \tau^{\sigma-\alpha}(\tau/t_m)^{\alpha +1}.
\end{align*}

For the remaining integrals,
without loss of generality, we shall assume that $t_m^*>t_1$
(as, otherwise, $\mathcal{I}_2=0$, while $\mathcal{I}_3$ is estimated with minimal changes).
Thus,
a different change of variable $\hat s:=s/t_m$  yields
\beq\label{I2_bound}
\mathcal{I}_2
{}\lesssim t_m^{\sigma-\alpha}\,(\tau/t_m)^{2/r}
\underbrace{
\int_{\tau/t_m}^{1-\theta}
(1 -\hat s)^{-\alpha -1}\,\hat s^{\sigma-2/r} d\hat s
}_{{}\lesssim\, (\tau/t_m)^{\sigma-2/r+1}+\theta^{-\alpha}}
\eeq
and
\beq\label{I3_bound}
\mathcal{I}_3\lesssim
t_m^{\sigma-\alpha}\,(\tau/t_m)^{1/r}
\underbrace{
\int_{1-\theta}^{1}
(1 -\hat s)^{-\alpha }\,\hat s^{\sigma-1-1/r} d\hat s
}_{{}\lesssim\, \theta^{1-\alpha}}\,,
\eeq
where, importantly, the bounds on the integrals in variable $\hat s$ rely on $\theta\le\frac12$.
Finally, combining the above bounds on $\mathcal{I}_1$, $\mathcal{I}_2$, and $\mathcal{I}_3$
with $r^m\lesssim
\mathcal{I}_1+\mathcal{I}_2+\mathcal{I}_3$, and
recalling that
$\theta=\frac12(\tau/t_m)^{1/r}$, one arrives at
\begin{align*}
r^m&{}\lesssim
\tau^{\sigma-\alpha}(\tau/t_m)^{\alpha +1}+
 t_m^{\sigma-\alpha}\,
\bigr[(\tau/t_m)^{\sigma+1}
+
(\tau/t_m)^{(2-\alpha)/r}\bigr]
\\&{}\lesssim
\tau^{\sigma-\alpha}
\bigl[
(\tau/t_m)^{\alpha +1}+
(\tau/t_m)^{(2-\alpha)/r+\alpha-\sigma}
\bigr].
\end{align*}
As $(\tau/t_m)\le 1$, the obtained bound is equivalent to the desired assertion~\eqref{r_bound_simple}.
\end{proof}

\begin{remark}[bound \eqref{trunc_er_B} in view of Lemma~\ref{lem_r}]\label{rem_trunc}
The truncation error bound \eqref{trunc_er_B} (used in the proof of Lemma~\ref{lem_stability1})
 can be viewed as a particular case
of a more general bound obtained in the proof of Lemma~\ref{lem_r}.
To be more precise, \eqref{trunc_er_B} bounds the truncation error of $B(s\,;t_p)$ multiplied by $S_J$.
In view of \eqref{B_L1_def}, the bound \eqref{trunc_er_B} reduces to $S_J\,|r^m|$, where $r^m$ is as above,
only with the integrand on $(0,t_p)$ vanishing (as $B(s\,;t_p)$ is linear on this interval).
So we proceed with the bound on $r^m$ as before,
only now $\sigma:=-\beta=\alpha-1$, so $t_m^{\sigma-\alpha}=t_m^{-1}$,
and
 $|r^m |{}\lesssim \mathcal{I}_2+\mathcal{I}_3$ (with $\mathcal{I}_1{}$ now vanishing).
For $\mathcal{I}_2{}$ one again gets \eqref{I2_bound}, only with the lower limit $\tau/t_m$ replaced by $t_p/t_m$,
while for  $\mathcal{I}_3{}$ one gets \eqref{I3_bound}.
Combining these observations yields \eqref{trunc_er_B}.
\end{remark}

\subsection{Truncation error for the semilinear term}

Next, we estimate the truncation error associated with the discretization of $f(u)$ by $F(U^{m},U^{m-1})$.

\begin{lemma}[truncation error for $f$]\label{lem_r_f}
Under the conditions of Lemma~\ref{lem_r} and condition A1 on $F(\cdot,\cdot)$, one has
\beq\label{r_f}
\bigl|r^m_f\bigr|:=\bigl|F(u(t_m),\,u(t_{m-1}))-f(u(t_m))\bigr|
\lesssim
\bigl\{t_m^{\hat\sigma}\,(\tau/t_m)^{1/r}\bigr\}^q,\quad \hat\sigma:=\min\{\sigma,\,1\}.
\eeq
If $q=2$, then $|r^m_f|$ enjoys the same bound as $|r^m|$ in \eqref{r_bound_simple}.
\end{lemma}

\begin{proof}
(i)
In view of A1, one gets $|r^m_f|\le L |u(t_m)-u(t_{m-1})|^q$, so
$|r^m_f|^{1/q}\lesssim \int_{t_{m-1}}^{t_m} |u'(s)|\,ds$.
If $\sigma\in(1,2)$, by~\eqref{ass_u_ode}, $|u'(s)|\lesssim 1$ so $|r^m_f|^{1/q}\lesssim \tau_m$,
while if $\sigma\in(0,1)$, a calculation yields $|r^m_f|^{1/q}\lesssim t_m^{\sigma-1}\tau_m$.
Combining these two cases yields $|r^m_f|^{1/q}\lesssim t_m^{\hat\sigma-1}\tau_m$
Finally, recall from~\eqref{t_grid_gen} that $\tau_m/t_m\lesssim (\tau/t_m)^{1/r}$, so
$|r^m_f|^{1/q}\lesssim t_m^{\hat\sigma}\,(\tau/t_m)^{1/r}$, which yields the desired assertion~\eqref{r_f}.

(ii)
It remains to show that for $q=2$ the bound on $|r^m|$ from \eqref{r_bound_simple} dominates the bound on $|r^m_f|$ in~\eqref{r_f}, i.e.
\beq\label{r_dominates}
\tau^{\sigma-\alpha}(\tau/t_m)^{\gamma+1}\gtrsim t_m^{2\hat\sigma}\,(\tau/t_m)^{2/r}.
\eeq
As the latter can be rewritten as $\tau^{p_1}t_m^{p_2}\gtrsim 1$ with some $p_1,\,p_2\in \R$,
clearly, it suffices to check that \eqref{r_dominates} holds true for $t_m=\tau$ and $t_m=T\simeq 1$.
This, respectively, is equivalent to
$\tau^{\sigma-\alpha}\gtrsim \tau^{2\hat\sigma}$
and
$\tau^{\sigma-\alpha+\gamma+1}\gtrsim \tau^{2/r}$.
Both hold true, the first one due to $\sigma-\alpha<\sigma\le2\hat\sigma$,
while the second follows from
$\sigma-\alpha+\gamma+1<\sigma-\alpha+\frac{2-\alpha}r+\alpha-\sigma=\frac{2-\alpha}r<2/r$.
\end{proof}

\begin{remark}[Case $q=1$]\label{rem_p_1}
Using an argument
similar to the one in part (ii) of the above proof, one can show that if $q=1$, then the bound on $|r^m|$ in \eqref{r_bound_simple} does not  dominate
$|r^m_f|\lesssim
t_m^{\hat\sigma}\,(\tau/t_m)^{1/r}\simeq \tau^{\hat\sigma}\,(\tau/t_m)^{1/r-\hat\sigma}$
(for which one would need $\sigma-\alpha\le \hat\sigma$
and $\sigma-\alpha+\gamma+1\le 1/r$, with the latter never satisfied).
On the other hand, \eqref{main_stab_new} readily provides a bound $\tau^{\hat\sigma}\,\U^j(\tau;\hat\gamma)$ on the possible additional error induced by $r^m_f$, with $\hat\gamma=1/r-\hat\sigma-1\le -\hat\sigma< 0$.
Hence, a calculation shows that  
the additional error, corresponding to $r^m_f$, is bounded by
$$
t_m^{\min\{\sigma,\,1\}+\alpha}\,(\tau/t_m)^{1/r}\simeq M^{-1}\, t_m^{\min\{\sigma,\,1\}+\alpha-1/r},
$$
where we used $\tau^{1/r}\simeq M^{-1}$
(in view of  \eqref{t_grid_gen}). Thus, in positive time one always gets only the first-order convergence
of this additional error component.
\end{remark}

\subsection{Pointwise-in-time error bound}

\begin{theorem}[pointwise-in-time error bound]\label{the_er_simplest}
Suppose that $u$ is a unique solution  of \eqref{simplest_a}, which satisfies \eqref{ass_u_ode} with some $\sigma \in (0,1) \cup (1,2)$,
while the numerical scheme \eqref{simplest_b} for \eqref{simplest_a}
satisfies
conditions A1 with $q=2$ and  A2 on $F(\cdot,\cdot)$, with some constants $\lambda_0,\lambda_1\ge 0$.
Additionally, suppose that
the temporal grid satisfies~\eqref{t_grid_gen} {\color{blue}{with some $ r\geq1$}} and
 $\lambda_0\tau_j^{\alpha}< \{\Gamma(2-\alpha)\}^{-1}$ $\forall\,j\ge1$.
Then there exists a unique solution  $\{U^m\}$ of \eqref{simplest_b}, and $\forall\,m\ge 1$
\beq\label{E_cal_m}
|u(t_m)-U^m|\lesssim {\mathcal E}^m\!:=\!
\left\{\!\!\begin{array}{ll}
M^{-\nu r}\,t_m^{\alpha-1}&\mbox{if~}1\le r<\frac{2-\alpha}{\nu},\\[0.2cm]
 M^{-(2-\alpha)}\,t_m^{\alpha-1}\,\ell_0
&\mbox{if~}r=\frac{2-\alpha}{\nu}\color{blue}\ge1,\\[0.1cm]
M^{-(2-\alpha)}\,t_m^{\sigma-\frac{2-\alpha}{r}}
&\mbox{if~}r>\frac{2-\alpha}{\nu},
\end{array}\right.    
\nu:=1+\sigma-\alpha,
\eeq
where $\ell_0=\ell_0(t_m)=[1+\ln(t_m/t_1)]$.

Under the above assumptions, only with $q=1$ in A1,
there exists a unique solution  $\{U^m\}$ of \eqref{simplest_b}, and
$|u(t_m)-U^m|\lesssim\widetilde{\mathcal E}^m := {\mathcal E}^m+M^{-1}\, t_m^{\min\{\sigma,\,1\}+\alpha-1/r}$
$\forall\,m\ge 1$.
\end{theorem}

\begin{remark}[{\color{blue}$ {\mathcal E}^m$ for $\sigma\in(1,2)$}]\color{blue}
If $\sigma\in(1,2)$, then $\frac{2-\alpha}{\nu}
<1$, so \eqref{E_cal_m} reduces to the case $r\ge 1>\frac{2-\alpha}{\nu}$.
Thus, ${\mathcal E}^m=M^{-(2-\alpha)}\,t_m^{\sigma-\frac{2-\alpha}{r}}$ for any $r\ge 1$.
From this, one immediately gets ${\mathcal E}^m\simeq M^{-(2-\alpha)}$ in positive time $t_m\gtrsim 1$ (which agrees with Corollary~\ref{cor_t1_con} below,
since $\nu r\ge (2-\alpha)\cdot 1$ implies $\min\{\nu r,\,2-\alpha\}=2-\alpha$).
\end{remark}

\begin{corollary}[convergence in positive time]\label{cor_t1_con}
For $t_m\gtrsim 1$ the error bound in \eqref{E_cal_m} involves
${\mathcal E}^m\simeq M^{-\min\{\nu r,\,2-\alpha\}}$ if $r\neq\frac{2-\alpha}{\nu}$
and ${\mathcal E}^m\simeq M^{-(2-\alpha)}\ln M$ otherwise.
Thus, $r>\frac{2-\alpha}{\nu}$ yields the optimal convergence rate with  ${\mathcal E}^m\simeq M^{-(2-\alpha)}$ in positive time.
\end{corollary}

\begin{corollary}[global convergence]\label{cor_glob_con}
The global-in-time version of \eqref{E_cal_m} is
\beq\label{E_cal_m_global}
\max_{t_m\le T}|u(t_m)-U^m|\lesssim \max_{m\le M}{\mathcal E}^m\simeq M^{-\min\{\sigma r,\,2-\alpha\}}.
\eeq
Thus, $\color{blue}r\ge \max\bigl\{\frac{2-\alpha}{\sigma},\,1\bigr\}$ yields the optimal global convergence rate $M^{-(2-\alpha)}$.
\end{corollary}

\begin{proof}
\color{blue}
If $\sigma r\ge 2-\alpha$, one gets the case $r>\frac{2-\alpha}{\nu}$ in \eqref{E_cal_m},
so ${\mathcal E}^m\lesssim M^{-(2-\alpha)}$ $\forall\, m\ge 1$.
Otherwise,  ${\mathcal E}^m$ involves a negative power of $t_m$, so ${\mathcal E}^m\lesssim{\mathcal E}^1$,
and a calculation (using $t_1\simeq M^{-r}$) shows that ${\mathcal E}^1\lesssim M^{-\sigma r}$ in each of the three cases in \eqref{E_cal_m}.
\end{proof}

\begin{remark}[earlier results as particular cases of \eqref{E_cal_m}]\label{rem_earlier_sigma}
The error bound~\eqref{E_cal_m} includes the following earlier results as particular cases:
for $\sigma = \alpha$ and $r \geqslant 1$ it is obtained using similar barrier-function techniques in \cite{Kopteva_Meng,kopteva_semilin}, while for
uniform temporal meshes (i.e. for $r=1$)  it is established using discrete Gr\"{o}nwall inequalities in
\cite{DLi1}.
The global bound \eqref{E_cal_m_global} for the particular case $\sigma=2\alpha$ is found in \cite[Theorem 4]{Gracia_2alpha}
for a fitted version of the L1 method (in which the error induced by $t^\alpha$ vanishes).
\end{remark}

\begin{proof}[Proof of Theorem~\ref{the_er_simplest}]
Recall the representation \eqref{delta_def_kappa} for the operator $\delta^\alpha_t$.
Note that $\lambda_0<\kappa_{m,m}$ $\forall\,m\ge1$.
Hence, given $\{U^j\}_{j<m}$, equation \eqref{simplest_b} is a nonlinear equation for $U^m$ of type
\beq\label{GUm}
G(U^m):=(\kappa_{m,m}-\lambda_0)U^m +\bigl\{F(U^m,w)+\lambda_0U^m\bigr\}=S,
\eeq
where $w= U^{m-1}$ and $S=\sum_{j=0}^{m-1}\kappa_{m,j}U^j+g(t_m)$ are in $\R$.
In view of A2, the function $G$ is continuous and strictly increasing, with $\lim_{v\rightarrow\pm\infty}G(v)=\pm\infty$, so the above equation has a unique solution $U^m\in\R$.

Next, consider the error $e^m:=U^m-u^m$, where we use the notation $u^m:=u(t_m)$ for the exact solution values at grid points.
The associated truncation errors $r^m=\delta_{t}^{\alpha} u^m-\pt_t^\alpha u(t_m)$ and $r_f^m=F(u^m,u^{m-1})-f(u^m)$ are
respectively estimated in Lemmas~\ref{lem_r} and~\ref{lem_r_f}.
Combining those with \eqref{simplest_a}, one gets
$\delta_t^\alpha u^m+F(u^m,u^{m-1})=r^m+r^{m}_f$.
Subtracting the latter from \eqref{simplest_b} yields
$e^0=0$ and
\beq\label{error_eq_simple}
\delta_t^\alpha e^m+F(u^m+e^m, U^{m-1})-F(u^m,u^{m-1})=-r^m-r^m_f\quad \forall\,m\ge 1.
\eeq
Multiply this equation by $\varsigma^m:={\rm sign}(e^m)$, equal to $1$ if $e^m\ge 1$ and $-1$ otherwise,
and note that $\varsigma^me^m=|e^m|$ so
\begin{subequations}\label{error_lead}
\begin{align}
\varsigma^m(\delta_t^\alpha e^m)\ge\kappa_{m,m}|e^m|-\sum_{j=0}^{m-1}\underbrace{\kappa_{m,j}}_{>0} |e^{j}|&= \delta_t^\alpha |e^m|,
\\
\varsigma^m\bigl[F(u^m+e^m, U^{m-1})-F(u^m,U^{m-1})\bigr]&\ge -\lambda_0 |e^m|,
\\
\intertext{while}
\bigl|F(u^m, U^{m-1})-F(u^m, u^{m-1})\bigr|&\le \lambda_1 |e^{m-1}|,
\end{align}
\end{subequations}
where we used \eqref{delta_def} and A2 (with the second relation following from the monotonicity of $F(v, U^{m-1})+\lambda_0 v$ in $v$, in view of A2).
Thus, \eqref{error_eq_simple} multiplied by $\varsigma^m$ yields
\beq\label{error_ineq}
(\delta_t^\alpha -\lambda_0)\, |e^m|-\lambda_1|e^{m-1}|\le |r^m|+|r^m_f|\qquad \forall\,m\ge 1.
\eeq

To complete the proof, it remains to recall the bounds of Lemmas~\ref{lem_r} and~\ref{lem_r_f} respectively for $r^m$ and $r^m_f$, and then apply Theorem~\ref{theo_main_stab_semi} to get  \eqref{E_cal_m}.
The cases of A1 with $q=2$ and $q=1$  will be considered separately.

First, consider A1 with $q=2$,
when $|r^m|+|r^m_f|\lesssim \tau^{\sigma - \alpha} (\tau/t_m)^{\gamma+1}$,
with
$\gamma:=\min\bigl\{\alpha,\,\,\frac{2-\alpha}{r}-\nu\bigr\}$.
So Theorem~\ref{theo_main_stab_semi} yields
$|e^j|\lesssim
\widetilde\U^j
:=\tau^{\sigma - \alpha}\,\U^j$, i.e.
$\widetilde\U^j:=
\ell_\gamma\tau^\nu t_j^{\alpha-1}(\tau/t_j)^{\min\{0,\,\gamma\}}$, and it suffices to show that $\widetilde\U^j\simeq {\mathcal E}^j$.
As $\alpha>0$, one has $\min\{0,\,\gamma\}=\min\{0,\,\frac{2-\alpha}{r}-\nu\}$, and
$\ell_\gamma\neq 1$ iff $r\neq\frac{2-\alpha}{\nu} $ (which is equivalent to $\gamma\neq 0$).
Now, $r<\frac{2-\alpha}{\nu}$
yields
$\min\{0,\,\gamma\}=0$, so $\widetilde\U^j\simeq \tau^\nu t_j^{\alpha-1}\simeq M^{-\nu r}t_j^{\alpha-1}\simeq {\mathcal E}^j$.
Similarly,  $r=\frac{2-\alpha}{\nu}$ yields $\widetilde\U^j
\simeq \ell_0 M^{-\nu r}t_j^{\alpha-1}\simeq {\mathcal E}^j$.
Finally,  $r>\frac{2-\alpha}{\nu}$ yields
 $\min\{0,\,\gamma\}=\frac{2-\alpha}{r}-\nu<0$, so
 $\widetilde\U^j\simeq \tau^\nu t_j^{\alpha-1}(\tau/t_j)^{\frac{2-\alpha}{r}-\nu}\simeq \tau^{\frac{2-\alpha}{r}}t_j^{\sigma-\frac{2-\alpha}{r}}$
so again $\widetilde\U^j\simeq M^{-(2-\alpha)} t_j^{\sigma-\frac{2-\alpha}{r}}\simeq {\mathcal E}^j$.

For the case $q=1$ in A1, the error induced by  $r^m$ is estimated similarly and remains $\lesssim {\mathcal E}^m$, while  the error induced by $r^m_f$
is estimated as described in Remark~\ref{rem_p_1}.
\end{proof}

\begin{remark}[{\color{blue}case $f=f(t,u)$}]\label{rem_ftu}
\color{blue}
The above analysis easily generalizes for
 a more general semilinear term $f(t,u)$ and its discretization $F(t_m,U^m,U^{m-1})$ in \eqref{simplest},
 as long as
 $F(t,\cdot,\cdot)$ satisfies
 assumptions A1 and A2 for each fixed $t\in[0,T]$ uniformly in $t$ (i.e. with $t$-independent constants $L$, $\lambda_0$, and $\lambda_1$).
 Indeed, then Lemma 3.3 and Theorem 3.5 remain valid with minimal changes in the proofs.
\end{remark}

\section{Error analysis for semilinear time-fractional parabolic equations}\label{sec_semi}
In this section we shall consider discretizations of type \eqref{semi_semidiscr_method} for the semilinear problem~\eqref{problem}.
It will be assumed throughout the section that the exact solution satisfies
\eqref{ass_u_pde}, i.e. that
$\|\partial_t^l u (\cdot, t)\|_{L_p(\Omega)}\lesssim 1+t^{\sigma-l}$ for $l = 0,1,2$ and $t\in(0,T]$, where $p=2$ or $p=\infty$,
with some $\sigma \in (0,1) \cup (1,2)$.
Our purpose is to establish pointwise-in-time bounds of type~\eqref{E_cal_m} in the spatial norms $L_2(\Omega)$ and $L_\infty(\Omega)$.
Both semidiscretizations in time and  full discretizations, using finite differences and finite elements, will be addressed,
with extensions  to full discretizations carried out along the lines of \cite[\S\S6-7]{kopteva_semilin}.

\subsection{Semidiscretization in time}

\begin{theorem}\label{theo_semidiscr}
Given $p\in\{2,\infty\}$, suppose that $u$ is a unique solution  of \eqref{problem},\eqref{LL_def} with the initial condition $u_0\in L_\infty(\Omega)$,
which satisfies
$u(\cdot,t)\in H^1_0(\Omega)$ for $t\in(0,T]$ and
\eqref{ass_u_pde}
with some $\sigma \in (0,1) \cup (1,2)$.
Let the numerical scheme \eqref{semi_semidiscr_method} for \eqref{problem}
satisfy
conditions A1 with $q=2$ and  A2 on $F(\cdot,\cdot)$, with some constants $\lambda_0,\lambda_1\ge 0$.
Also, let
the temporal grid satisfy~\eqref{t_grid_gen} and
 $\lambda_0\tau_j^{\alpha}< \{\Gamma(2-\alpha)\}^{-1}$ $\forall\,j\ge1$.
Then there exists a unique solution  $\{U^m\}$ of \eqref{semi_semidiscr_method},
 with $U^m\in H^1_0(\Omega)\cap L_\infty(\Omega)$ 
and
\beq\label{L1_semi_error}
\|u(\cdot,t_m)-U^m\|_{L_p(\Omega)}\lesssim {\mathcal E}^m
\qquad\forall\,m=1,\ldots,M, 
\eeq
where ${\mathcal E}^m$ is from \eqref{E_cal_m}.

Under the above assumptions, only with $q=1$ in A1, the error bound \eqref{L1_semi_error} is replaced by
$\|u(\cdot,t_m)-U^m\|_{L_p(\Omega)}\lesssim\widetilde{\mathcal E}^m={\mathcal E}^m+M^{-1}\, t_m^{\min\{\sigma,\,1\}+\alpha-1/r}$
$\forall\,m\ge 1$.
\end{theorem}

\begin{proof}
We combine the proof of Theorem~\ref{the_er_simplest} with the proof of \cite[Theorem~5.1]{kopteva_semilin}.

The existence of a unique solution $U^m\in H^1_0(\Omega)\cap L_\infty(\Omega)$ can be shown by induction using \cite[Lemma~2.1(i)]{kopteva_semilin}.
Indeed, assuming that there exist desired $\{U^j\}_{j<m}$, for $U^m$ one gets an elliptic-equation version
$\LL U^m+G(x,U^m)=S$
of \eqref{GUm},
where the same definition of $\color{blue}G=G(x,U^m)$ involves $w= U^{m-1}(x)\in L_\infty(\Omega)$, while $S=\sum_{j=0}^{m-1}\kappa_{m,j}U^j\in L_\infty(\Omega)$.
To apply \cite[Lemma~2.1(i)]{kopteva_semilin}
{\color{blue}(which relies on \cite[Lemma~1]{Demlow_Kopteva} and \cite[Lemma~16]{Brezis_Strauss})}
to this equation,
note that A2 yields the one-sided Lipschitz condition on $F$ in $v$ with the constant $\lambda_0$,
the continuity of  $F(v,w)$ in $v$ {\color{blue}(so $G(x,v)$ is continuous in $v$),}  as well as
$F(v,U^{m-1}(\cdot))\in L_\infty(\Omega)$ $\forall\,v\in\R$
{\color{blue}(so $G(\cdot,v)\in L_\infty(\Omega)$ $\forall\,v\in\R$)}.

It remains to establish~\eqref{L1_semi_error}, for which it suffices to get a version of \eqref{error_ineq}, in which each
term of type $|\cdot|$ is replaced by the corresponding
 $\|\cdot\|_{L_p(\Omega)}$:
\beq\label{error_ineq_semi}
(\delta_t^\alpha -\lambda_0)\, \|e^m\|_{L_p(\Omega)}-\lambda_1\|e^{m-1}\|_{L_p(\Omega)}\le \|r^m\|_{L_p(\Omega)}+\|r^m_f\|_{L_p(\Omega)}\quad \forall\,m\ge 1.
\eeq
Indeed, once \eqref{error_ineq_semi} is true, using
\eqref{ass_u_pde},
one gets bounds of type \eqref{r_bound_simple} and \eqref{r_f} 
on $\|r^m\|_{L_p(\Omega)}$ and $\|r^m_f\|_{L_p(\Omega)}$,
and then applies Theorem~\ref{theo_main_stab_semi} to get  \eqref{L1_semi_error}
exactly as in the proof of Theorem~\ref{the_er_simplest}.

If $p=2$, one easily gets \eqref{error_ineq_semi} by taking the $L_2(\Omega)$ inner product
 (denoted $\langle\cdot,\cdot\rangle$) of $e^m$
with the following version of \eqref{error_eq_simple}:
\beq\label{err_eq_semi}
\delta_t^\alpha e^m +\LL e^m+F(u^m+e^m, U^{m-1})-F(u^m,u^{m-1})=-r^m-r^m_f\quad\;\; \forall\,m\ge 1.
\eeq
Then, \eqref{error_ineq_semi}, indeed, follows from
$\langle\LL e^m,e^m\rangle\ge0$ (in view of $c-\frac12\sum_{k=1}^d\pt_{x_k}\!b_k\ge 0$) combined with
 the following version of \eqref{error_lead} (in which $\|\cdot\|=\|\cdot\|_{L_2(\Omega)}$):
\begin{align*}
\bigl\langle\delta_t^\alpha e^m,\,e^m\bigr\rangle \ge\kappa_{m,m}\|e^m\|^2-\sum_{j=0}^{m-1}\underbrace{\kappa_{m,j}}_{>0} \|e^{j}\|\,\|e^m\|&=
\bigl(\delta_t^\alpha \|e^m\|\bigr)\,\|e^m\|,
\\
\bigl\langle F(u^m+e^m, U^{m-1})-F(u^m,U^{m-1}),\,e^m\bigr\rangle&\ge -\lambda_0 \|e^m\|^2,
\\[0.3cm]
\bigl|\bigl\langle F(u^m, U^{m-1})-F(u^m, u^{m-1}),\,e^m\bigr\rangle\bigr|&\le \lambda_1 \|e^{m-1}\|\,\|e^m\|.
\end{align*}

Finally, if $p=\infty$, one gets \eqref{error_ineq_semi}
by multiplying \eqref{err_eq_semi}
with
$\varsigma^m:={\rm sign}(e^m)\in L_\infty(\Omega)$
and then employing \eqref{error_lead}, which yields an elliptic-equation version of
\eqref{error_ineq} (in which $\|\cdot\|=\|\cdot\|_{L_\infty(\Omega)}$):
\beq\label{error_ineq}
\varsigma^m\,(\LL+\kappa_{m,m})e^m
\le \mu^m:=
\!\sum_{j=0}^{m-1}\!\kappa_{m,j} \|e^{j}\|
+\lambda_0 \|e^m\|
+\lambda_1\|e^{m-1}\|+ \|r^m\|+\|r^m_f\|.\!
\eeq
The above equation holds $\forall\,x\in\Omega$ (and $\forall\,m\ge 1$). If $\|e^m\|=|e^m(x^*)|=(\varsigma^m\, e^m)(x^*)$ for some $x^*\in\Omega$, and $\LL e^m(x^*)$ is defined
in the classical sense, then $(\varsigma^m\,\LL e^m)(x^*)\ge 0$, so \eqref{error_ineq} implies $\kappa_{m,m}\|e^m\|\le \mu^m$, which immediately implies
the desired \eqref{error_ineq_semi}
for $p=\infty$.
The same result also follows from \eqref{error_ineq}
for less smooth $e^m\in H^1_0(\Omega)\cap L_\infty(\Omega)$
(note that $\LL e^m \in L_\infty(\Omega)$),
as described in Remark~\ref{rem_weak_max}.
\end{proof}

\begin{remark}[weak maximum principles for functions in $H^1_0(\Omega)$]\label{rem_weak_max}
For weak maximum principles for functions in $H^1_0(\Omega)$, we refer the reader to
\cite[\S8.1]{GTru}.
For our purposes, the following result should be applied to \eqref{error_ineq}:\\
If
$v\in H_0^1(\Omega)\cap L_\infty(\Omega)$ and  $\LL v\in L_\infty(\Omega)$, while $\varsigma:={\rm sign}(v)$,
then, with any positive constants $\kappa$ and $\mu$,
$$
\varsigma\, \bigl(\LL+\kappa\bigr)v\le \mu\quad\Rightarrow\quad \kappa\,\|v\|_{L_\infty}\le \mu.
$$
Indeed, consider $w:=( v-\mu/\kappa)^+=\max\{ v-\mu/\kappa\,,0\}$, which is in $H_0^1(\Omega)$ \cite[Lemma~7.6]{GTru}.
Assuming this function has non-empty support $\Omega^+:=\{\kappa v>\mu\}\subseteq\Omega$,
one gets $(\LL+\kappa)[w+\mu/\kappa]=(\LL+\kappa)v\le \mu$ in $\Omega^+$. Combining this with $(\LL+\kappa) [\mu/\kappa]=(\mu/\kappa)\,\LL[1]+\mu \ge \mu$ yields
$(\LL+\kappa)w\le 0$ in $\Omega^+$.
Taking the inner $L_2(\Omega^+)$ product of the latter bound with $w\ge 0$, and noting that
 $\int_{\Omega^+}\!w\,\LL w\ge 0$ (for any $w\in H^1_0(\Omega^+)$, one gets $\int_{\Omega^+}\!\kappa w^2\le 0$. This implies that $\Omega^+$ is empty, so $\kappa v\le \mu$ in $\Omega$.
Similarly,
rewriting  $\text{sign}(v)\, (\LL+\kappa)v\le \mu$ as
$\text{sign}(-v)\, (\LL+\kappa)[-v]\le \mu$ yields $\kappa(-v)\le \mu$.
The desired assertion $\kappa\|v\|_{L_\infty(\Omega)}\le \mu$ follows.
\end{remark}

\begin{remark}[{\color{blue}case $f=f(x,t,u)$}]\label{rem_fxtu}
\color{blue}
Similarly to Remark~\ref{rem_ftu}, the above analysis extends to a more general $f=f(x,t,u)$ in \eqref{problem} and
its discretization
$F=F(x,t_m,U^m,U^{m-1})$ in \eqref{semi_semidiscr_method}.
For this case, we assume that
$F(x,t,v,w)$ satisfies
 assumptions A1 and A2 uniformly in $x\in\Omega$ and $t>0$.
 Furthermore, a version of A2 for this case requires $F$ to be continuous in $v$ and $w$,
 and $F(\cdot,t,v,w)\in L_\infty(\Omega)$ $\forall\, t>0$ and $\forall\,v,w\in\R$.
 Under these assumptions, Theorem~\ref{theo_semidiscr} remains valid with minimal changes in the proofs.
\end{remark}

\subsection{Maximum norm error for finite difference discretizations}\label{ssec_FD}

In  the spatial domain $\Omega=(0,1)^d\subset\R^d$,
let
 $\bar\Omega_h$ be the tensor product of $d$ uniform meshes $\{ih\}_{i=0}^N$,
 with $\Omega_h:=\bar\Omega_h\backslash\pt\Omega$ denoting the set of interior mesh nodes.
Now, consider a finite difference discretization in space of \eqref{semi_semidiscr_method}
in the form
\beq\label{FD_method}
\delta_t^\alpha U^m(z) +\LL_h U^m(z)+ F(U^m(z),U^{m-1}(z))=0\quad\mbox{for}\;\;z\in\Omega_h
\eeq
subject to $U^m=0$ in $\bar\Omega_h\cap\pt\Omega$ $\forall\,m\ge 1$ and $U^0=u_0$ in $\bar\Omega_h$.
 Here the discrete spatial operator $\LL_h$
is a standard finite difference operator defined,
using the standard orthonormal basis $\{\mathbf{i}_k\}_{k=1}^d$ in $\R^d$
(such that $z=(z_1,\ldots,z_d)=\sum_{k=1}^d z_k\, \mathbf{i}_k$ for any $z\in\R^d$), by
\begin{align*}
&\LL_hV(z):=\\[-0.1cm]
&\sum_{k=1}^d h^{-2}\Bigl\{a_k(z+{\textstyle \frac12}h\mathbf{i}_k)\,\bigl[V(z)-V(z+h\mathbf{i}_k)\bigr]+a_k(z-{\textstyle \frac12}h\mathbf{i}_k)\,\bigl[V(z)-V(z-h\mathbf{i}_k)\bigr]\Bigr\}\\[-0.4cm]
&\qquad\quad{}+\sum_{k=1}^d{\textstyle \frac12}h^{-1}\, b_k(z)\,\bigl[V(z+h\mathbf{i}_k)-V(z-h\mathbf{i}_k)\bigr] +c(z)\,V(z)
\quad\qquad\mbox{for}\;\;z\in\Omega_h.
\end{align*}
(Here the terms in the first and second sums respectively discretize $-\pt_{x_k}\!(a_k\,\pt_{x_k}\!u)$ and $b_k\, \pt_{x_k}\!u$ from \eqref{LL_def}.)

Assuming that $h$ is sufficiently small, to be more precise,
\beq\label{d_max_pr}
h^{-1}\ge \max_{k=1,\ldots,d}\bigl\{{\textstyle\frac12}\|b_k\|_{L_\infty(\Omega)}\,\|a_k^{-1}\|_{L_\infty(\Omega)}\bigr\},
\eeq
and that $c\ge 0$ in $\Omega\times(0,T]$,
one can easily check that  the spatial discrete operator $\LL_h$ is associated with an M-matrix,
while $\LL_h $ satisfies the discrete maximum principle.
Then, \cite[Lemma~2.5(i)]{kopteva_semilin} yields the existence of a unique solution $\{U^j\}_{j=0}^m$ for the fully discrete problem~\eqref{FD_method}.


Furthermore, under the conditions of Theorem~\ref{theo_semidiscr}, additionally assuming \eqref{d_max_pr}
and
$\|\pt_{x_k}^l u(\cdot,t)\|_{L_\infty(\Omega)}\lesssim 1$ for $l=3,4$, $k=1,\ldots,d$, $t\in(0,T]$,
one gets a version of~\eqref{L1_semi_error}
for the nodal maximum norm of the error
$$
\max_{\Omega_h}|u(\cdot,t_m)-U^m|
\lesssim {\mathcal E}^m+t_m^{\alpha}\, h^2
$$
if $q=2$, while
${\mathcal E}^m$ in the above bound is replaced by $\widetilde{\mathcal E}^m$
if $q=1$.

The proof of the above error bounds, to a large degree, imitates the proof of Theorem~\ref{theo_semidiscr} for $p=\infty$.
Additionally, one needs to take into account the truncation error $O(h^2)$ induced by the finite difference discretization $\LL_h$ of $\LL$,
which is estimated using the same stability result~\eqref{main_stab_new} only with $\gamma=-1$.

\subsection{Finite element discretizations}

Let $\LL$ in \eqref{problem} be
$\LL:=-\triangle u=-\sum_{k=1}^d\pt^2_{x_k}$
(i.e. $a_k=1$, $b_k=0$ for $k=1,\ldots,d$ and  $c=0$ in \eqref{LL_def}).
A full discretization of \eqref{problem}, posed in a general bounded  Lipschitz domain  $\Omega\subset\R^d$,
can be defined
 by applying
 a standard finite element spatial approximation to the temporal semidiscretization~\eqref{semi_semidiscr_method} as follows.
 Let $S_h \subset H_0^1(\Omega)\cap C(\bar\Omega)$ be a Lagrange finite element space of fixed degree $\ell\ge 1$ 
 relative to a
 quasiuniform simplicial triangulation
 $\mathcal T$ of $\Omega$
 (assuming, to simplify the presentation, that the triangulation covers $\Omega$ exactly).
 Now,
 let $u_h^0= u_0$, and
 for $m=1,\ldots,M$, let $u^m_h \in S_h$ satisfy
 \beq\label{FE_problem}
\begin{array}{l}
\langle \delta_t^\alpha u_h^m,v_h\rangle +\langle\nabla u_h^m,\nabla v_h\rangle+ \langle F(u_h^m, u_h^{m-1}),v_h\rangle=0\qquad\forall v_h\in S_h\,.
\end{array}
\eeq
%

To extend Theorem~\ref{theo_semidiscr}
to  \eqref{FE_problem}, we  proceed along the lines of \cite[\S7]{kopteva_semilin} (see also \cite{NK_MC_L1})
and employ the standard Ritz projection $\RR_h u(t)\in S_h$ of $u(\cdot,t)$
defined by
$\langle \nabla\RR_h u,\nabla v\rangle=\langle -\triangle u,v_h\rangle$ 
$\forall v_h\in S_h$,
$t\in[0,T]$.
A stronger version of the one-sided Lipschitz condition  A2 on $F(\cdot,\cdot)$ will be assumed 
%
 with some constants $\lambda_0\ge0$ and $\lambda_1\ge 0$:
\beq\label{F_ass_A2star}
|F(v+\nu, w+\omega)-F(v, w)|\le \lambda_0 |\nu|+\lambda_1|\omega| \quad \forall\,v,\,w\in\RR_u\subseteq \R,\;\; \forall\,\nu,\,\omega\in\R.
\eeq

Under the conditions of Theorem~\ref{theo_semidiscr} (with $p=2$ in \eqref{ass_u_pde}) and condition
\eqref{F_ass_A2star}
(instead of A2) on $F(\cdot,\cdot)$,
there exists a unique solution $\{u_h^j\}_{j=0}^M$ of \eqref{FE_problem}, and, for $m=1,\ldots,M$,
\beq
\|u(\cdot,t_m)-u_h^m\|_{L_2(\Omega)}\lesssim
{\mathcal E}^m
 \label{FE_error_bound_L2}
+\max_{t\in[0,t_m]}\|\rho(\cdot, t)\|_{L_2(\Omega)}+\int_0^{t_m}\!\|\pt_t \rho(\cdot, t)\|_{L_2(\Omega)}\,dt,
\eeq
where
 $\rho(\cdot, t):=\RR_h u(t)-u(\cdot, t)$
 is the error of the Ritz projection,
and
${\mathcal E}^m$ is from \eqref{E_cal_m} assuming $q=2$ in A1, while if  $q=1$, the term ${\mathcal E}^m$ is replaced by
$\widetilde{\mathcal E}^m$.

Indeed, the existence of a unique solution $u_h^m$ for each $m\ge1$ can be proved exactly as in the proof of \cite[Theorem~7.4]{kopteva_semilin}.
Next,
let $e_h^m:=u_h^m-\RR_h u(t_m)\in S_h$ and $\rho^m:=\rho(\cdot, t_m)$, so
 $u_h^m=u^m+e_h^m+\rho^m$.
Now, a standard calculation using \eqref{FE_problem}  and \eqref{problem} yields  a version of \eqref{err_eq_semi}:
$$
\langle \delta_t^\alpha e_h^m, v_h\rangle +\langle\nabla e_h^m,\nabla v_h\rangle
+\langle F(u_h^m, u_h^{m-1})-F(u^m,u^{m-1}), v_h\rangle=-\langle \delta_t^\alpha \rho^m+r^m+r^m_f, v_h\rangle
$$
$\forall\, v_h\in S_h$.
Setting $v_h:=e_h^m$ and noting that, by \eqref{F_ass_A2star},
$|F(u_h^m, u_h^{m-1})-F(u^m,u^{m-1})|$ is bounded by $\lambda_0|e_h^m+\rho^m|+\lambda_1|e_h^{m-1}+\rho^{m-1}|$,
one gets a version of \eqref{error_ineq_semi} (in which $\|\cdot\|=\|\cdot\|_{L_2(\Omega)}$):
$$
(\delta_t^\alpha -\lambda_0)\, \|e^m\|-\lambda_1\|e^{m-1}\|\le \lambda_0 \|\rho^m\|+\lambda_1\|\rho^{m-1}\|+\|\delta_t^\alpha \rho^m+r^m+r^m_f\|
$$
for $m\ge 1$, subject to $e_h^0=-\rho^0$.
The desired bound \eqref{FE_error_bound_L2} is now obtained along the lines of the proof of \cite[Theorem~7.4]{kopteva_semilin}.

\section{Quasilinear time-fractional subdiffusion equations}\label{sec_quasi}

Since the coefficients in \eqref{problem} are functions of $x$ and $t$,
our error analysis framework naturally extends to the quasilinear case, as we  now demonstrate.

Consider a quasilinear version of~\eqref{problem},\eqref{CaputoEquiv}:
\beq\label{problem_q}
\begin{array}{l}
\pt_t^{\alpha}u+\QQ u+f(x,t,u)=0\quad\mbox{for}\;\;(x,t)\in\Omega\times(0,T],\\[0.2cm]
u(x,t)=0\quad\mbox{for}\;\;(x,t)\in\pt\Omega\times(0,T],\qquad
u(x,0)=u_0(x)\quad\mbox{for}\;\;x\in\Omega.
\end{array}
\eeq
Here the quasilinear spatial operator $\QQ$ is defined by
\beq\label{LL_quasi}
\QQ u := -\sum_{k=1}^d \pt_{x_k}\!\Bigl\{a_k(x,t,u)\,\pt_{x_k}\!u + b_k(x,t,u) \Bigr\},
\eeq
under the following conditions
\beq\label{C_a}
\begin{array}{rl}
0<c_a\le a_k(x,t,v)&\le \bar c_a,
\quad
\\[0.2cm]
|a_k(x,t,v)-a_k(x,t,w)|&\le C_a|v-w|,
\quad
\\[0.2cm]
|b_k(x,t,v)-b_k(x,t,w)|&\le C_b|v-w|,
\quad\forall x\in\Omega,\;\;t>0 ,\;\;  v,\,w\in\R,
\end{array}
\eeq
with some constants $\bar c_a\ge c_a>0$, $C_a\ge 0$, and $C_b\ge0$.

Let $f$ in \eqref{problem_q} be continuous in $s$ and  satisfy $f(\cdot,t,s)\in L_\infty(\Omega)$ for all $t>0$ and $s\in\R$,
and the one-sided Lipschitz condition
\beq\label{ass_f}
f(x,t,v)-f(x,t,w)\ge -\lambda_0[v-w]\qquad \forall v\ge w,\;\; x\in\Omega,\;\;t>0
\eeq
with some constant {\color{blue}{$\lambda_0\ge0$}}.
If $f=f(u)$, the above \eqref{ass_f} becomes equivalent to A1 and A2 for $F(v,w):=f(v)$ with constants $L=0$ and $\lambda_1=0$.
Note also that \eqref{ass_f} was assumed in \cite{kopteva_semilin} when considering semilinear subdiffusion equations.

To simplify the presentation, we shall restrict our consideration to the fully implicit semidiscretization
of \eqref{problem_q} in time of L1 type (compare with \eqref{semi_semidiscr_method}):
\beq\label{semi_semidiscr_method_quasi}
\delta_t^\alpha U^m +\QQ^m U^m+ f(\cdot,t_m,U^m)=0\quad\mbox{in}\;\Omega,
\eeq
$\forall\,m\ge 1$, subject to  $U^m=0$\ on $\pt\Omega$ and $U^0=u_0$,
where $\QQ^m=\QQ\big|_{t=t_m}$.
For each $m$, equation~\eqref{semi_semidiscr_method_quasi} is a quasilinear elliptic equation.
Existence (and uniqueness) of solutions of such equations is typically established by the application of topological fixed point theorems
in appropriate function spaces
\cite{GTru,evans}, as we illustrate in \S\ref{ssec_existence}.

\begin{theorem}[error bound]\label{theo_semidiscr_quasi}
Suppose that problem \eqref{problem_q},\,\eqref{LL_quasi},
under conditions \eqref{C_a},\,\eqref{ass_f}, has a unique solution
$u$, which satisfies
\eqref{ass_u_pde}
for $p=2$ and some constant $\sigma \in (0,1) \cup (1,2)$,
and  $\|\nabla u(\cdot,t)\|_{L_\infty(\Omega)}\le C_u$ $\forall\, t\in[0,T]$  for some constant $C_u$.
%
Let
the temporal grid satisfy~\eqref{t_grid_gen} and
 $\lambda\tau_j^{\alpha}< \{\Gamma(2-\alpha)\}^{-1}$ $\forall\,j\ge1$,
 where $\lambda:=\lambda_0 +{(C_aC_u+C_b\sqrt{d})^2}/(4c_a)$.
Then a solution  $U^m\in H^1_0(\Omega)$ of \eqref{semi_semidiscr_method_quasi},
satisfies
\beq\label{L1_q_error}
\|u(\cdot,t_m)-U^m\|_{L_2(\Omega)}\lesssim {\mathcal E}^m
\qquad\forall\,m=1,\ldots,M, 
\eeq
where ${\mathcal E}^m$ is from \eqref{E_cal_m}.
%
\end{theorem}

\begin{proof}
Throughout the proof, let $\|\cdot\|=\|\cdot\|_{L_2(\Omega)}$.
For the error
$e^m=U^m-u^m$
we again get a version of \eqref{err_eq_semi} $\forall\,m\ge 1$, in which $\LL e^m$ is now replaced by $\QQ^m U^m -\QQ^m u^m$
(and also $F(v,\cdot)$ is replaced by $f(\cdot,t_m, v)$, and $r^{m}_f=0$):
\beq\label{err_eq_quasi}
\delta_t^\alpha e^m +\bigl\{\QQ^m U^m-\QQ^m u^m\bigr\}
+f(\cdot,t_m,u^m+e^m)-f(\cdot,t_m,u^m)
=-r^m
\eeq
To be more precise, \eqref{err_eq_quasi} is obtained by subtracting $\delta_t^\alpha u^m+\QQ^m u^m
+f(\cdot,t_m,u^m)=r^m
$
from \eqref{semi_semidiscr_method_quasi}.
Next, taking  the $L_2(\Omega)$ inner product
 of $e^m$
with~\eqref{err_eq_quasi},
one gets (similarly to deriving \eqref{error_ineq_semi} for $p=2$  in the proof of Theorem~\ref{theo_semidiscr})
\beq\label{er_aux_q}
(\delta_t^\alpha -\lambda_0)\, \|e^m\|
+\frac{\langle \QQ^m U^m-\QQ^m u^m, e^m\rangle}{\|e^m\|}
\le\|r^m\|\,.
\eeq
Here, in view of \eqref{LL_quasi},
$$
\langle \QQ^m U^m-\QQ^m u^m, e^m\rangle
=\sum_{k=1}^d  \langle -\pt_{x_k} w^m_k \,, e^m\rangle
=\sum_{k=1}^d  \langle w^m_k \,, \pt_{x_k} e^m\rangle,
$$
where
$$
w^m_k=a_k(\cdot,t_m,U^m)\,\pt_{x_k}\!U^m- a_k(\cdot,t_m,u^m)\,\pt_{x_k}\!u^m
+b_k(\cdot,t_m,U^m)-b_k(\cdot,t_m,u^m).
$$
A calculation using \eqref{C_a} shows that
$$
w^m_k\ge a_k(\cdot,t_m,U^m)\,\pt_{x_k}\!e^m-C_a\,|e^m|\,|\pt_{x_k}\!u^m|-C_b|e^m|.
$$
Hence, with  $\|\nabla u^m\|_{L_\infty(\Omega)}\le C_u$, one gets
\begin{align*}
\langle \QQ^m U^m-\QQ^m u^m, e^m\rangle&=\sum_{k=1}^d  \langle w^m_k \,, \pt_{x_k} e^m\rangle
\\[2pt]
&\ge c_a\|\nabla e^m\|^2-C_aC_u \|\nabla e^m\|\,\|e^m\|- C_b\sqrt{d}\,\|\nabla e^m\|\,\|e^m\|
\\[2pt]
&\ge -\frac{(C_aC_u+C_b\sqrt{d})^2}{4c_a}\|e^m\|^2.
\end{align*}
Combining this with \eqref{er_aux_q}, one arrives at
$$
(\delta_t^\alpha -\lambda)\, \|e^m\|
\le \|r^m\|,
\quad\mbox{where}\;\;
\lambda=\lambda_0 +\frac{(C_aC_u+C_b\sqrt{d})^2}{4c_a}.
$$
Now, using
\eqref{ass_u_pde}, one gets a bound of type \eqref{r_bound_simple}
on $\|r^m\|$,
and then applies Theorem~\ref{theo_main_stab_semi} to get  \eqref{L1_q_error}
exactly as in the proof of Theorem~\ref{the_er_simplest}.
\hfill
\end{proof}

\subsection{Existence and uniqueness for \eqref{semi_semidiscr_method_quasi} with $a_k=a(u)$}\label{ssec_existence}
Here we shall assume, similarly to \cite{Jin2024,Plociniczak2023}, that
 the diffusion coefficients $a_k:=a(u)$ $\forall\,k$, while still allowing a non-self-adjoint case of $\QQ$ in \eqref{LL_quasi}.
Note also that the following existence and uniqueness result seamlessly applies to the case $a_k=a(t,u)$.

\begin{lemma}\label{lem_exist_quasi}
Under conditions \eqref{C_a},\,\eqref{ass_f},
let
the temporal grid satisfy
 $\lambda\tau_j^{\alpha}< \{\Gamma(2-\alpha)\}^{-1}$ $\forall\,j\ge1$,
 where $\lambda$ is from Theorem~\ref{theo_semidiscr_quasi},
and suppose that $U^0\in L_2(\Omega)$. Then there exists a unique solution $U^m\in H^1_0(\Omega)$ of \eqref{semi_semidiscr_method_quasi} $\forall\,m\ge 1$.
\end{lemma}

\begin{proof}
\newcommand{\TTT}{\mathbb T}
Set $A(w):=\int_0^w a(s)\,ds$; then 
$\QQ w = -\sum_{k=1}^d \pt_{x_k}\!\bigl\{\pt_{x_k}\!A(w) + b_k(x,t,w) \bigr\}$.
Now, assuming the desired assertion is proved for $\{U^j\}_{j<m}$,
and also recalling \eqref{delta_def_kappa}, one sees that
 \eqref{semi_semidiscr_method_quasi} is a quasilinear elliptic equation for $U^m$, any solution of which is
 a fixed point of the nonlinear mapping $\TTT$ in $L_2(\Omega)$, defined by $w=\TTT v$, where
\beq\label{T_def_A}
-\sum_{k=1}^d \pt_{x_k}^2\!A(w) + [\kappa_{m,m}w+f(\cdot,t_m,w)]=\sum_{k=1}^d \pt_{x_k}\! b_k(x,t,v) +S,
\eeq
subject to $w=0$ on $\pt\Omega$,
with $S=\sum_{j<m}\kappa_{m,j}U^j\in L_2(\Omega)$.
Taking the difference of \eqref{T_def_A} for any $w_1=\TTT v_1$ and $w_2=\TTT v_2$, and then the inner product of the latter with $A(w_1)-A(w_2)=:\omega[w_1-w_2]$,
yields
$$
\|\nabla(\omega[w_1-w_2])\|^2+(\kappa_{m,m}-\lambda_0)c_a\|w_1-w_2\|^2\le C_b\sqrt{d}\,\|v_1-v_2\|\,\|\nabla(\omega[w_1-w_2])\|.
$$
Here we used
the standard linearization of $A(w_1)-A(w_2)$ yielding $\omega=\int_0^1 a(sw_1+[1-s]w_2)\,ds\in L_\infty(\Omega)$
(as its range is within $[c_a,\bar c_a]$) and
$\omega\ge c_a>0$.
Hence, a calculation shows that
\beq\label{contraction}
(\kappa_{m,m}-\lambda_0)\,c_a\|w_1-w_2\|^2\le\frac{C_b^2 d}{4}\|v_1-v_2\|^2.
\eeq
The definition of $\kappa_{m,m}$ in \eqref{delta_def_kappa},
combined with the restriction on $\tau_m$,
implies that $\kappa_{m,m}<\lambda$, while  $\lambda=\lambda_0 +{(C_aC_u+C_b\sqrt{d})^2}/(4c_a)$.
Thus $\TTT$ is a strict contraction, so, by Banach's fixed point theorem \cite[\S9.2.1]{evans}, one concludes that $\TTT$ indeed has a unique fixed point
in $L_2(\Omega)$, which, in view of \eqref{T_def_A}, is also in $H^1_0(\Omega)$.

Importantly, the above argument relies on the semilinear equation \eqref{T_def_A} having a unique solution in $H_0^1(\Omega)$ for any $v, S\in L_2(\Omega)$.
The uniqueness is straightforward from \eqref{contraction} with $v_1=v_2$.
The existence
of $w\in H_0^1(\Omega)$
can be shown, for example,
applying the method of  sub- and super-solutions \cite[\S9.3]{evans}
to \eqref{T_def_A} rewritten for $W:=A(w)$, with the left-hand side
$-\sum_{k=1}^d \pt_{x_k}^2 W+\hat f(\cdot, A^{-1}(W))$,
where
$\hat f(\cdot,w):=\kappa_{m,m}w +f(\cdot,t_m,w)$ is increasing in $w$,
and
$A^{-1}:\R\rightarrow \R$ denotes the inverse of $A:\R\rightarrow \R$.
\end{proof}

\section{Numerical results}\label{sec_Num}
In this section, out theoretical findings are illustrated by numerical experiments with
three test problems.


\begin{figure}[h!]
\label{Fig1}
\begin{center}
\includegraphics[height=0.31\textwidth]{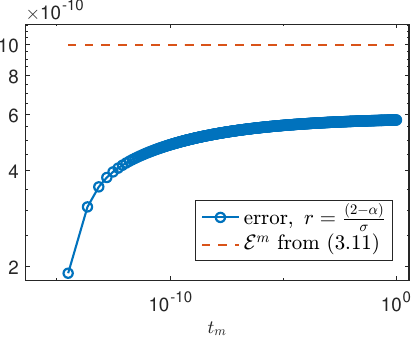}\hspace{0.05\textwidth}~~\includegraphics[height=0.30\textwidth]{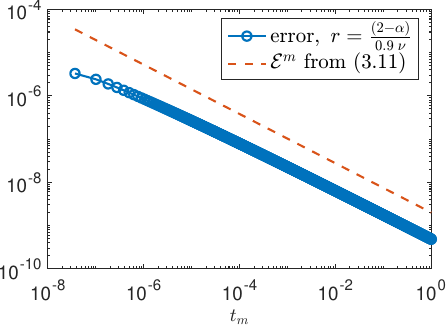}
\end{center}
\vspace{-0.3cm}
\caption{Test problem A: pointwise errors for \eqref{simplest_b} with $\alpha = 0.3$, $\sigma = 2\alpha$, $M = 2^{17}$, $r=\frac{2-\alpha}{\sigma}$
(left) and $r=\frac{2-\alpha}{0.9\nu}$ (right), \color{blue}{where $\nu:= 1+\sigma-\alpha$ }.}
 \end{figure}

\subsection*{Test problem A}
We start with a simplest test \eqref{simplest_a}
without spatial derivatives, with $f(u) = 0$, $g(t)= \frac{\Gamma(\sigma+1)}{\Gamma(\sigma-\alpha+1)}t^{\sigma-\alpha}$, $u_0 =0$, and $T=1$, which has an exact solution $u(t) = t^\sigma$. Fig.\,\ref{Fig1} displays the pointwise errors for scheme \eqref{simplest_b} (with $F=0$), demonstrating their agreement with the theoretical bounds of Theorem \ref{the_er_simplest}.

\begin{figure}[h!]
\label{Fig2}
\begin{center}
\includegraphics[height=0.31\textwidth]{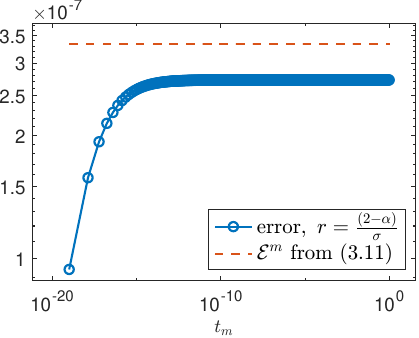}\hspace{0.05\textwidth}~~\includegraphics[height=0.30\textwidth]{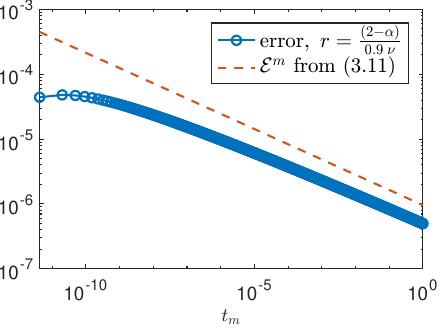}
\end{center}
\vspace{-0.3cm}
\caption{\color{blue}{Test problem A: pointwise errors for \eqref{simplest_b}
with $\alpha = 0.7$, $\sigma = \alpha/2$, $M = 2^{17}$, $r=\frac{2-\alpha}{\sigma}$
(left) and $r=\frac{2-\alpha}{0.9\nu}$ (right), where $\nu:= 1+\sigma-\alpha$.}}
 \end{figure}

\subsection*{Test problem B}
Next, consider the semilinear problem \eqref{problem} in $\Omega = (0,\pi)^2$
with {\color{blue}{$\mathcal{L} = -(\partial_{x_1}^2+\partial_{x_2}^2)$}}, $f(u) = u^3-u +g(x,t)$,  $u_0= 0$, and $T=1$. We choose $g$
such that the exact solution is given by {\color{blue}{$u(x,t) = t^\sigma\sin({x_1}^2/\pi)\sin({x_2}^2/\pi)$}}.
Fig.\,\ref{Example B} shows the maximum nodal errors for scheme \eqref{semi_semidiscr_method}
using the first-order IMEX
and the second-order Newton-iteration IMEX schemes
(with $F(U^m, U^{m-1})$, respectively, from Examples 3 and 4 in \S\ref{sec:intro}).
The standard finite difference discretization from \S\ref{ssec_FD} is used in space.
We consider $\alpha = 0.3$ and $\alpha = 0.5$,
$\sigma = \alpha/2$ and $r = \frac{2-\alpha}{\sigma}$.
For these values, Theorem \ref{theo_semidiscr} implies global error bounds
$\max_{\{t_m\}}\mathcal{E}^m\simeq M^{-(2-\alpha)}$
and $\max_{\{t_m\}}\widetilde{\mathcal{E}}^m\simeq M^{-1}$,
which are both in good agreement with the errors in Fig.\,\ref{Example B}.

\begin{figure}[h!]
\label{Example B}
\begin{center}
\includegraphics[height=0.34\textwidth]{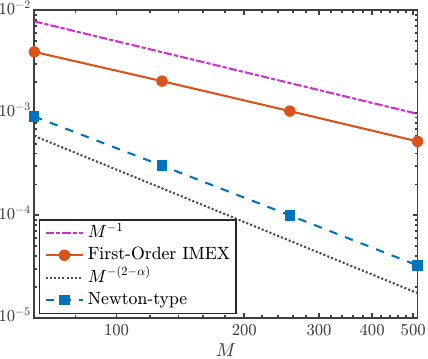}\hspace{0.05\textwidth}~~\includegraphics[height=0.34\textwidth]{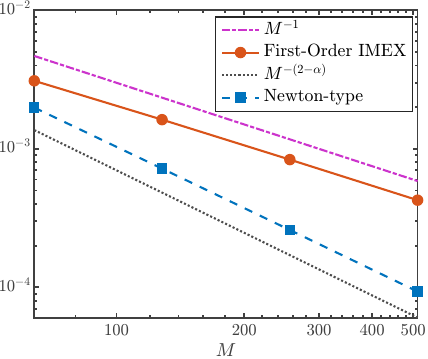}
\end{center}
\vspace{-0.3cm}
\caption{\color{blue}{Test problem B: maximum nodal errors for numerical scheme (\ref{semi_semidiscr_method} (first-order IMEX and second-order Newton-iteration IMEX versions) with $2^{10}$ degrees of freedom in space, $\sigma = \alpha/2$, $r = \frac{2-\alpha}{\sigma}$,
$\alpha = 0.3$ (left) and $\alpha = 0.5$ (right).}}
 \end{figure}

\subsection*{Test problem C}
Now, we consider the quasi-linear equation \eqref{problem_q} posed in $\Omega = (0,1)$,
with
$\mathcal{Q} u = -\partial_{x}\!\big\{(1 + u)\,\partial_{x}u\;\!\big\}$, $f(x,t,u) = u(1-u)$,
$u_0 = x(1-x)$, and $T=1$, known as a subdiffusive Fisher-Kolmogorov equation.
{\color{blue}{This equation was also considered in \cite{Plociniczak2023} under the assumption that
$\sigma = \alpha$ (while the errors listed \cite[Table~1]{Plociniczak2023} agree with Corollaries~\ref{cor_glob_con} and~\ref{cor_t1_con}
for the particular case $\sigma=\alpha$ and $r=1$).}
As we also want to numerically test the case $\sigma<\alpha$, we consider a modification of the above problem,
 {\bf Test problem C$\mathbf{}'$}, which is obtained by including an additional function $g(x,t):= \frac{\Gamma(\sigma+1)}{\Gamma(\sigma-\alpha+1)}t^{\sigma-\alpha}$
 (used in Test problem~A to get a more singular solution) in the right side of the corresponding \eqref{problem_q}. }

A standard second-order finite difference discretization (similar to the one considered in \S\ref{ssec_FD}) is used in space. {\color{blue}{Since the exact solution is unavailable for Test problems C and C${}'$,  the maximal nodal errors, computed using the double mesh approach,
are listed in Table \ref{tableC}
together with the corresponding computational convergence rates for a graded temporal mesh. The table confirms}} the expected convergence rate $2-\alpha$ (which follows from Theorem~\ref{theo_semidiscr_quasi} combined with Corollary \ref{cor_glob_con}).

\begin{table}
\label{tableC}
\caption{Test problems C (left) and the modification $C'$ (right) : Maximum nodal errors  (odd rows) and computational rates $q$ in $M^{-q}$ (even rows), $2^{13}$ spatial degrees of freedom, {\color{blue}{obtained using double-mesh comparison.}}}
\vspace{-0.2cm}
\tabcolsep=3pt
\centering

\begin{tabular}{l llll  llll}

\hline\rule{0pt}{11pt}
  & \multicolumn{4}{l}{Test~C:$\;\;r = {(2-\alpha)/}{\alpha}$}
  & \multicolumn{4}{l}{Test~C${}'$:$\;\;r= {(2-\alpha)/}{\sigma}$, $\;\sigma = \alpha/2$} \vspace{0.1cm}\\

\hline
\vspace{0.1cm}
  & \strut\rule{0pt}{10pt}$M = 2^7$ & $M = 2^{8}$ & $M = 2^{9}$ & $M = 2^{10}$
  & \strut\rule{0pt}{9pt}$M = 2^7$ & $M = 2^{8}$ & $M = 2^{9}$ & $M = 2^{10}$ \\
\hline
\rule{0pt}{10pt}
$\alpha=0.3\;$
& 2.26e-4 & 7.63e-5 & 2.53e-5 & 8.30e-6
& 1.78e-4 & 5.95e-5 & 1.97e-5 & 6.46e-6 \\
& 1.565 & 1.592 & 1.609 &
& 1.577 & 1.595 & 1.611 & \\[3pt]

$\alpha=0.5$
& 4.86e-4 & 1.87e-4 & 6.92e-5 & 2.52e-5
& 4.35e-4 & 1.60e-4 & 5.79e-5 & 2.09e-5 \\
& 1.379 & 1.434 & 1.456 &
& 1.445 & 1.465 & 1.472 & \\[3pt]

$\alpha=0.7$
& 9.21e-4 & 4.17e-4 & 1.80e-4 & 7.61e-5
& 1.17e-3 & 5.03e-4 & 2.12e-4 & 8.76e-5 \\
& 1.144 & 1.213 & 1.241 &
& 1.213 & 1.251 & 1.272 & \\

\hline
\end{tabular}
\end{table}

\smallskip

While here we have presented limited numerical results for most representative cases,
we have observed similarly good agreement with the theoretical error bounds in
all our numerical experiments for wide ranges of $\alpha$, $\sigma$, and $r$,
both for pointwise-in-time  and global errors.

  \bibliographystyle{siamplain}
  \bibliography{lit}

@article {DLi1,
    AUTHOR = {Li, Dongfang and Qin, Hongyu and Zhang, Jiwei},
     TITLE = {Sharp pointwise-in-time error estimate of {L}1 scheme for nonlinear subdiffusion equations},
   JOURNAL = {J. Comput. Math.},
  FJOURNAL = {Journal of Computational Mathematics},
    VOLUME = {42},
      YEAR = {2024},
    NUMBER = {3},
     PAGES = {662--678},
      ISSN = {0254-9409,1991-7139},
   MRCLASS = {65M12 (35K20 35R11 65M60)},
  MRNUMBER = {4729054},
MRREVIEWER = {Ruilian\ Du},

}

@article {Kopteva_Meng,
    AUTHOR = {Kopteva, Natalia and Meng, Xiangyun},
     TITLE = {Error analysis for a fractional-derivative parabolic problem
              on quasi-graded meshes using barrier functions},
   JOURNAL = {SIAM J. Numer. Anal.},
  FJOURNAL = {SIAM Journal on Numerical Analysis},
    VOLUME = {58},
      YEAR = {2020},
    NUMBER = {2},
     PAGES = {1217--1238},
      ISSN = {0036-1429},
   MRCLASS = {65M60 (35R11 65M15)},

}

@article {kopteva_semilin,
    AUTHOR = {Kopteva, Natalia},
     TITLE = {Error analysis for time-fractional semilinear parabolic
              equations using upper and lower solutions},
   JOURNAL = {SIAM J. Numer. Anal.},
  FJOURNAL = {SIAM Journal on Numerical Analysis},
    VOLUME = {58},
      YEAR = {2020},
    NUMBER = {4},
     PAGES = {2212--2234},
      ISSN = {0036-1429,1095-7170},
   MRCLASS = {65M60 (35R11 65M15)},
  MRNUMBER = {4129005},

}

@article {delay_cmam,
    AUTHOR = {Cen, Dakang and Vong, Seakweng},
     TITLE = {The tracking of derivative discontinuities for delay
              fractional equations based on fitted {$L1$} method},
   JOURNAL = {Comput. Methods Appl. Math.},
  FJOURNAL = {Computational Methods in Applied Mathematics},
    VOLUME = {23},
      YEAR = {2023},
    NUMBER = {3},
     PAGES = {591--601},
      ISSN = {1609-4840,1609-9389},
}

@article {NK_MC_L1,
    AUTHOR = {Kopteva, Natalia},
     TITLE = {Error analysis of the {L}1 method on graded and uniform meshes
              for a fractional-derivative problem in two and three
              dimensions},
   JOURNAL = {Math. Comp.},
  FJOURNAL = {Mathematics of Computation},
    VOLUME = {88},
      YEAR = {2019},
    NUMBER = {319},
     PAGES = {2135--2155},
      ISSN = {0025-5718},
   MRCLASS = {65M06 (35R11 65M15 65M60)},

}

@article {NK_AML_L2nonl,
    AUTHOR = {Kopteva, Natalia},
     TITLE = {Error analysis of an L2-type method on graded meshes for semilinear subdiffusion equations},
   JOURNAL = {Appl. Math. Lett.},
    VOLUME = {160},
      YEAR = {2025},
    NUMBER = {109306},
}

@book {GTru,
    AUTHOR = {Gilbarg, David and Trudinger, Neil S.},
     TITLE = {Elliptic partial differential equations of second order},
    SERIES = {Classics in Mathematics},
      NOTE = {Reprint of the 1998 edition},
 PUBLISHER = {Springer-Verlag, Berlin},
      YEAR = {2001},
     PAGES = {xiv+517},
      ISBN = {3-540-41160-7},
   MRCLASS = {35-02 (35Jxx)},
  MRNUMBER = {1814364},
}

@book {evans,
    AUTHOR = {Evans, Lawrence C.},
     TITLE = {Partial differential equations},
    comment = {SERIES = {Graduate Studies in Mathematics}},
    comment = {VOLUME = {19}},
 PUBLISHER = {American Mathematical Society, Providence, RI},
      YEAR = {1998},
     PAGES = {xviii+662},
      ISBN = {0-8218-0772-2},
   MRCLASS = {35-01},

}

@book {Diet10,
    AUTHOR = {Diethelm, Kai},
     TITLE = {The analysis of fractional differential equations},
    comment = {SERIES = {Lecture Notes in Mathematics}},
    VOLUME = {2004},
 PUBLISHER = {Springer-Verlag, Berlin},
      YEAR = {2010},
     PAGES = {viii+247},
      ISBN = {978-3-642-14573-5},
   MRCLASS = {34-02 (26A33 33E12 34A08)},
  MRNUMBER = {2680847},

}

@article {Gracia_2alpha,
    AUTHOR = {Gracia, J. L. and O'Riordan, E. and Stynes, M.},
     TITLE = {A fitted scheme for a {C}aputo initial-boundary value problem},
   JOURNAL = {J. Sci. Comput.},
  FJOURNAL = {Journal of Scientific Computing},
    VOLUME = {76},
      YEAR = {2018},
    NUMBER = {1},
     PAGES = {583--609},
      ISSN = {0885-7474,1573-7691},
}

@article{
sorg17,
AUTHOR = {Stynes, Martin and O'Riordan, Eugene and Gracia, Jos\'e{}
              Luis},
     TITLE = {Error analysis of a finite difference method on graded meshes
              for a time-fractional diffusion equation},
   JOURNAL = {SIAM J. Numer. Anal.},
  FJOURNAL = {SIAM Journal on Numerical Analysis},
    VOLUME = {55},
      YEAR = {2017},
    NUMBER = {2},
     PAGES = {1057--1079}}

@article {JLZ19,
    AUTHOR = {Jin, Bangti and Lazarov, Raytcho and Zhou, Zhi},
     TITLE = {Numerical methods for time-fractional evolution equations with
              nonsmooth data: a concise overview},
   JOURNAL = {Comput. Methods Appl. Mech. Engrg.},
  FJOURNAL = {Computer Methods in Applied Mechanics and Engineering},
    VOLUME = {346},
      YEAR = {2019},
     PAGES = {332--358},
      ISSN = {0045-7825,1879-2138},
   MRCLASS = {65M60 (35R11 65M15)},
  MRNUMBER = {3894161},
}

@article {maskari19,
    AUTHOR = {Al-Maskari, Mariam and Karaa, Samir},
     TITLE = {Numerical approximation of semilinear subdiffusion equations
              with nonsmooth initial data},
   JOURNAL = {SIAM J. Numer. Anal.},
  FJOURNAL = {SIAM Journal on Numerical Analysis},
    VOLUME = {57},
      YEAR = {2019},
    NUMBER = {3},
     PAGES = {1524--1544},
      ISSN = {0036-1429,1095-7170},
   MRCLASS = {65M60 (65M12 65M15)},
  MRNUMBER = {3975154},
}

@article {jin18,
    AUTHOR = {Jin, Bangti and Li, Buyang and Zhou, Zhi},
     TITLE = {Numerical analysis of nonlinear subdiffusion equations},
   JOURNAL = {SIAM J. Numer. Anal.},
  FJOURNAL = {SIAM Journal on Numerical Analysis},
    VOLUME = {56},
      YEAR = {2018},
    NUMBER = {1},
     PAGES = {1--23},
      ISSN = {0036-1429,1095-7170},
   MRCLASS = {65M60 (35R11 45K05 65M12 65M15)},
  MRNUMBER = {3742688},
MRREVIEWER = {Georgios\ D.\ Akrivis},

}

@article {liao18,
    AUTHOR = {Liao, Hong-lin and Li, Dongfang and Zhang, Jiwei},
     TITLE = {Sharp error estimate of the nonuniform {L}1 formula for linear
              reaction-subdiffusion equations},
   JOURNAL = {SIAM J. Numer. Anal.},
  FJOURNAL = {SIAM Journal on Numerical Analysis},
    VOLUME = {56},
      YEAR = {2018},
    NUMBER = {2},
     PAGES = {1112--1133},
      ISSN = {0036-1429,1095-7170},
   MRCLASS = {65M06 (35B65 35R11 65M15)},
  MRNUMBER = {3790081},
MRREVIEWER = {Gabriella\ Bretti},

}

@article{
Du2020,
Author = {Qiang Du and Jiang Yang and Zhi Zhou},
Title = {Time-Fractional Allen–Cahn Equations: Analysis and Numerical Method},
Journal = {Journal of Scientific Computing},
Volume = {85},
Pages = {},
Year = {2020}}

@article{
Rasheed24,
Author = {Maan A. Rasheed and Maani A. Saeed},
Title = {Numerical Study for a Two-dimensional Time-fractional Semi linear Parabolic Equation Using Linearly Implicit
Euler Finite Difference Method with Caputo Derivative},
Journal = {AIP Conf. Proc},
Volume = {3036},

Pages = {},
Year = {2024}}

@article {HLiao1,
    AUTHOR = {Liao, Hong-lin and Yan, Yonggui and Zhang, Jiwei},
     TITLE = {Unconditional convergence of a fast two-level linearized
              algorithm for semilinear subdiffusion equations},
   JOURNAL = {J. Sci. Comput.},
  FJOURNAL = {Journal of Scientific Computing},
    VOLUME = {80},
      YEAR = {2019},
    NUMBER = {1},
     PAGES = {1--25},
      ISSN = {0885-7474,1573-7691},
   MRCLASS = {65M06 (65M12)},
  MRNUMBER = {3954434},
}

@article {Ji20,
    AUTHOR = {Ji, Bingquan and Liao, Hong-lin and Zhang, Luming},
     TITLE = {Simple maximum principle preserving time-stepping methods for
              time-fractional {A}llen-{C}ahn equation},
   JOURNAL = {Adv. Comput. Math.},
  FJOURNAL = {Advances in Computational Mathematics},
    VOLUME = {46},
      YEAR = {2020},
    NUMBER = {2},
     PAGES = {Paper No. 37, 24},
      ISSN = {1019-7168,1572-9044},
   MRCLASS = {65M06 (35B50 35R11 65M12 65M50)},
  MRNUMBER = {4083875},
}

@article {Plociniczak24,
    AUTHOR = {P{\l}ociniczak, {\L}ukasz and Ta\'zbierski, Kacper},
     TITLE = {Fully discrete {G}alerkin scheme for a semilinear subdiffusion
              equation with nonsmooth data and time-dependent coefficient},
   JOURNAL = {Comput. Math. Appl.},
  FJOURNAL = {Computers \& Mathematics with Applications. An International
              Journal},
    VOLUME = {165},
      YEAR = {2024},
     PAGES = {217--223},
      ISSN = {0898-1221,1873-7668},
   MRCLASS = {65M60 (65M12)},
  MRNUMBER = {4744094},
}

@article {Plociniczak2023,
    AUTHOR = {P{\l}ociniczak, {\L}ukasz},
     TITLE = {A linear {G}alerkin numerical method for a quasilinear
              subdiffusion equation},
   JOURNAL = {Appl. Numer. Math.},
  FJOURNAL = {Applied Numerical Mathematics. An IMACS Journal},
    VOLUME = {185},
      YEAR = {2023},
     PAGES = {203--220},
      ISSN = {0168-9274,1873-5460},
   MRCLASS = {65M60 (35R11 65R20)},
  MRNUMBER = {4516332},
}

@misc{Jin2024,
      title={Regularity Analysis and High-Order Time Stepping Scheme for Quasilinear Subdiffusion},
      author={Bangti Jin and Qimeng Quan and Barbara Wohlmuth and Zhi Zhou},
      year={2024},
      eprint={2407.19146},
      archivePrefix={arXiv},
      primaryClass={math.NA},
}

@article{Lopez2023,
  title={Convolution Quadrature for the quasilinear subdiffusion equation},
  author={L{\'o}pez-Fern{\'a}ndez, Maria and P{\l}ociniczak, {\L}ukasz},
  journal={arXiv preprint arXiv:2311.00081},
  year={2023}
}

@article {Demlow_Kopteva,
    AUTHOR = {Demlow, Alan and Kopteva, Natalia},
     TITLE = {Maximum-norm a posteriori error estimates for singularly
              perturbed elliptic reaction-diffusion problems},
   JOURNAL = {Numer. Math.},
  FJOURNAL = {Numerische Mathematik},
    VOLUME = {133},
      YEAR = {2016},
    NUMBER = {4},
     PAGES = {707-742},
      ISSN = {},
   MRCLASS = {},
  MRNUMBER = {},
MRREVIEWER = {},
       DOI = {},
       URL = {},
}

@article {Brezis_Strauss,
    AUTHOR = {Br\'ezis, Haim and Strauss, Walter A.},
     TITLE = {Semi-linear second-order elliptic equations in {L1}},
   JOURNAL = {J. Math. Soc. Japan},
  FJOURNAL = {},
    VOLUME = {25},
      YEAR = {1973},
    NUMBER = {},
     PAGES = {},
      ISSN = {},
   MRCLASS = {},
  MRNUMBER = {},
MRREVIEWER = {},
       DOI = {},
       URL = {},
}

\end{document}